\documentclass[12pt]{article}

\usepackage{amsmath, amsthm, amssymb, stmaryrd}
\usepackage[all]{xy}
\usepackage{fullpage}
\usepackage{titlesec}
\usepackage{mathrsfs}
\usepackage{amsbsy}
\usepackage{turnstile}
\usepackage{bbm}
\usepackage{yhmath}
\usepackage{tensor}
\usepackage{graphics}
\usepackage{enumitem}
\usepackage{calligra}
\usepackage{verbatim}
\usepackage{setspace}
\usepackage{hhline}

\usepackage[pdftex,bookmarks=true]{hyperref}
\usepackage[svgnames]{xcolor}
\hypersetup{
    linkbordercolor = {PaleTurquoise},
    citebordercolor = {PaleGreen},
}

\makeatletter
\newcommand*{\@old@slash}{}\let\@old@slash\slash
\def\slash{\relax\ifmmode\delimiter"502F30E\mathopen{}\else\@old@slash\fi}
\makeatother

\titleformat{\section}{\normalsize\bfseries}{\thesection}{1em}{}
\titleformat{\subsection}{\normalsize\bfseries}{\thesubsection}{1em}{}

\numberwithin{equation}{subsection}

\theoremstyle{plain}

\newtheorem{PropSub}[subsection]{Proposition}
\newtheorem{LemSub}[subsection]{Lemma}
\newtheorem{CorSub}[subsection]{Corollary}
\newtheorem{ThmSub}[subsection]{Theorem}

\theoremstyle{definition}

\newtheorem{DefSub}[subsection]{Definition}
\newtheorem{ExaSub}[subsection]{Example}
\newtheorem{RemSub}[subsection]{Remark}
\newtheorem{ParSub}[subsection]{}

\newtheorem*{Acknowledgement}{Acknowledgement}

\newcommand*{\emptybox}{\leavevmode\hbox{}}

\newdir{ >}{{}*!/-10pt/@{>}}

\DeclareMathAlphabet{\mathpzc}{OT1}{pzc}{m}{it}
\DeclareMathAlphabet{\mathcalligra}{T1}{calligra}{m}{n}

\newcommand{\bref}[1]{\textnormal{\ref{#1}}}
\newcommand{\pbref}[1]{\textnormal{(\ref{#1})}}

\newcommand{\A}{\ensuremath{\mathscr{A}}}
\newcommand{\B}{\ensuremath{\mathscr{B}}}
\newcommand{\C}{\ensuremath{\mathscr{C}}}
\newcommand{\D}{\ensuremath{\mathscr{D}}}

\newcommand{\normalJ}{\ensuremath{\mathscr{J}}}
\newcommand{\J}{\ensuremath{{\kern -0.4ex \mathscr{J}}}}

\newcommand{\JF}{\ensuremath{{\kern -0.4ex \mathscr{J}_{\kern -0.05ex F}}}}

\newcommand{\sP}{\ensuremath{\mathscr{P}}}
\newcommand{\Q}{\ensuremath{\mathscr{Q}}}
\newcommand{\R}{\ensuremath{\mathscr{R}}}
\newcommand{\sS}{\ensuremath{\mathscr{S}}}
\newcommand{\T}{\ensuremath{\mathscr{T}}}
\newcommand{\U}{\ensuremath{\mathscr{U}}}
\newcommand{\V}{\ensuremath{\mathscr{V}}}

\newcommand{\X}{\ensuremath{\mathscr{X}}}

\newcommand{\uV}{\ensuremath{\mkern2mu\underline{\mkern-2mu\mathscr{V}\mkern-6mu}\mkern5mu}}

\newcommand{\CAT}{\ensuremath{\operatorname{\textnormal{\text{CAT}}}}}

\newcommand{\VCAT}{\ensuremath{\V\textnormal{-\text{CAT}}}}

\newcommand{\NN}{\ensuremath{\mathbb{N}}}

\newcommand{\RR}{\ensuremath{\mathbb{R}}}
\newcommand{\SSS}{\ensuremath{\mathbb{S}}}
\newcommand{\TT}{\ensuremath{\mathbb{T}}}

\newcommand{\UU}{\ensuremath{\mathbb{U}}}

\newcommand{\ZZ}{\ensuremath{\mathbb{Z}}}

\newcommand{\ob}{\ensuremath{\operatorname{\textnormal{\textsf{ob}}}}}
\newcommand{\mor}{\ensuremath{\operatorname{\textnormal{\textsf{mor}}}}}

\newcommand{\Lan}{\ensuremath{\operatorname{\textnormal{\textsf{Lan}}}}}

\newcommand{\End}{\ensuremath{\operatorname{\textnormal{End}}}}

\newcommand{\kk}[1]{\textnormal{\textsf{k}}_{\scriptscriptstyle #1}}
\newcommand{\kt}[1]{\widetilde{\textnormal{\textsf{k}}}_{\scriptscriptstyle #1}}

\newcommand{\stt}{\mathbin{\tilde{*}}}

\newcommand{\ca}[1]{\scalebox{0.85}{\raisebox{0.3ex}{$|$}} #1\scalebox{0.85}{\raisebox{0.3ex}{$|$}}}

\newcommand{\Th}{\ensuremath{\textnormal{Th}}}

\newcommand{\ThJ}{\ensuremath{\Th_{\kern -0.5ex \normalJ}}}
\newcommand{\SubThJ}{\ensuremath{\textnormal{SubTh}_{\kern -0.5ex \normalJ}}}
\newcommand{\Vfp}{\ensuremath{\V_{\kern -0.5ex fp}}}

\newcommand{\Ev}{\ensuremath{\textnormal{\textsf{Ev}}}}
\newcommand{\Coev}{\ensuremath{\textnormal{\textsf{Coev}}}}

\newcommand{\inJ}[1]{#1 \in \normalJ\kern -1.2ex}

\newcommand{\VCATJ}{\VCAT_{\kern -0.7ex \normalJ}}
\newcommand{\PhiJ}{\Phi_{\kern -0.8ex \normalJ}}
\newcommand{\MndJ}{\Mnd_{\kern -0.8ex \normalJ}}

\newcommand{\Set}{\ensuremath{\operatorname{\textnormal{\text{Set}}}}}
\newcommand{\FinCard}{\ensuremath{\operatorname{\textnormal{\text{FinCard}}}}}

\newcommand{\Ab}{\ensuremath{\operatorname{\textnormal{\text{Ab}}}}}

\newcommand{\Mod}[1]{\ensuremath{#1\textnormal{-\text{Mod}}}}

\newcommand{\Mat}{\ensuremath{\textnormal{\text{Mat}}}}

\newcommand{\Mnd}{\ensuremath{\operatorname{\textnormal{\text{Mnd}}}}}

\newcommand{\Alg}[1]{\ensuremath{#1\kern -.5ex\operatorname{\textnormal{-\text{Alg}}}}}
\newcommand{\CAlg}[1]{\ensuremath{#1\kern -.5ex\operatorname{\textnormal{-\text{CAlg}}}}}

\newcommand{\CRProf}{\ensuremath{\operatorname{\textnormal{\text{CRProf}}}}}
\newcommand{\CRProfJ}{\ensuremath{\CRProf_{\kern -0.6ex \normalJ}}}

\newcommand{\Adjn}[6]{\xymatrix {#1 \ar@/_0.5pc/[rr]_{#2}^(0.4){#4}^(0.6){#5}^{\top} & & #6 \ar@/_0.5pc/[ll]_{#3}}}
\newcommand{\Equiv}[6]{\xymatrix {#1 \ar@/_0.5pc/[rr]_{#2}^(0.4){#4}^(0.6){#5}^{\sim} & & #6 \ar@/_0.5pc/[ll]_{#3}}}

\newcommand{\lt}{\leqslant}

\newcommand{\op}{\ensuremath{\textnormal{op}}}

\newcommand{\aff}{\ensuremath{\textnormal{aff}}}

\newcommand{\pushoutcorner}{\ar@{}[dr]|(.3)\ulcorner}
\newcommand{\pullbackcorner}{\ar@{}[dr]|(.3)\lrcorner}

\newcommand{\cmt}[1]{}

\begin{document}

\author{\normalsize  Rory B. B. Lucyshyn-Wright\thanks{The author gratefully acknowledges financial support in the form of an AARMS Postdoctoral Fellowship, a Mount Allison University  Research Stipend, and, earlier, an NSERC Postdoctoral Fellowship.}\let\thefootnote\relax
\\
\small Mount Allison University, Sackville, New Brunswick, Canada}

\title{\large \textbf{Commutants for enriched algebraic theories and monads}}

\date{}

\maketitle

\abstract{We define and study a notion of \textit{commutant} for $\V$-enriched $\J$-algebraic theories for a system of arities $\J$, recovering the usual notion of commutant or centralizer of a subring as a special case alongside Wraith's notion of commutant for Lawvere theories as well as a notion of commutant for $\V$-monads on a symmetric monoidal closed category $\V$.  This entails a thorough study of commutation and Kronecker products of operations in $\J$-theories.  In view of the equivalence between $\J$-theories and $\J$-ary monads we reconcile this notion of commutation with Kock's notion of commutation of cospans of monads and, in particular, the notion of commutative monad.  We obtain notions of $\J$-\textit{ary commutant} and \textit{absolute commutant} for $\J$-ary monads, and we show that for finitary monads on $\Set$ the resulting notions of finitary commutant and absolute commutant coincide.  We examine the relation of the notion of commutant to both the notion of \textit{codensity monad} and the notion of \textit{algebraic structure} in the sense of Lawvere.}

\section{Introduction} \label{sec:intro}

Given a pair of endomorphisms $\mu,\nu:S \rightarrow S$ in a category $\T$ we can ask whether $\mu$ and $\nu$ \textit{commute}, i.e. whether $\mu \cdot \nu = \nu \cdot \mu$.  Interestingly, this notion of commutation generalizes to apply to pairs of morphisms $\mu:S^J \rightarrow S^{J'}$ and $\nu:S^K \rightarrow S^{K'}$ between various \textit{powers} of a given object $S$ in a category $\T$, where $J,J',K,K'$ are sets.  Indeed, extrapolating from Linton's classic work \cite{Lin:AutEqCats}, the pair $\mu,\nu$ determines associated morphisms $\mu * \nu$, $\mu \stt \nu$ $:$ $S^{J \times K} \rightarrow S^{J' \times K'}$ that we call the \textit{first and second Kronecker products} of $\mu$ and $\nu$, and we say that $\mu$ and $\nu$ \textit{commute} if $\mu * \nu = \mu \stt \nu$ \pbref{def:commutation}.  The importance of this notion of commutation stems from the fact that mappings $S^J \rightarrow S$ defined on a power of a given set $S$ are fundamental to Birkhoff's \textit{universal algebra} \cite[II]{Bir:SelPa}, where they are called \textit{($J$-ary) operations}.  Classically, one restricts attention to operations whose \textit{arities} $J$ are finite cardinals.  It is an insight of Lawvere \cite{Law:PhD} that any variety of algebras in Birkhoff's sense is described by an abstract category $\T$, called an \textit{algebraic theory} or \textit{Lawvere theory}, whose objects are the finite powers $S^0,S^1,S^2,...$ of a single object $S = S^1$.  Individual algebras of the given variety are then described equivalently as $\T$-\textit{algebras}, i.e. functors $A:\T \rightarrow \Set$ that preserve finite powers.  For convenience one often takes the objects of $\T$ to be the finite cardinals $0,1,2,...\,$.  For example, left $R$-modules for a ring $R$ can be described as $\T$-algebras where $\T$ is a category whose morphisms are $R$-matrices, wherein the first Kronecker product $\mu * \nu$ of a pair of morphisms is the classical \textit{Kronecker product} $\nu \otimes \mu$ of the matrices $\nu$ and $\mu$ \cite[4.4]{Lu:CvxAffCmt}.

Given a subtheory $\T \hookrightarrow \U$ of a Lawvere theory $\U$, one can define the \textit{commutant} of $\T$ in $\U$ as the subtheory $\T^\perp \hookrightarrow \U$ consisting of all morphisms $\mu \in \mor\U$ such that $\mu$ commutes with every $\nu \in \mor\T$.  This notion of commutant was introduced briefly by Wraith \cite{Wra:AlgTh} and is studied further in the author's recent paper \cite{Lu:CvxAffCmt} with attention to specific examples of theories that arise as commutants.

In the present paper we study a generalization of this notion of commutant in the context of $\V$-enriched $\J$-{algebraic theories} for a system of arities $\J$ in the sense of \cite{Lu:EnrAlgTh}, obtaining notions of commutant for $\V$-\textit{monads} on $\V$ as special cases.  This entails a detailed study of several fundamental aspects of the theory of $\V$-enriched universal algebra for a system of arities $\J$, including commutation and Kronecker products of operations.

By definition, a \textit{system of arities} $\J \rightarrow \V$ in a symmetric monoidal closed category $\V$ is a fully faithful symmetric strong monoidal $\V$-functor.  Up to an equivalence, a system of arities is therefore simply a full sub-$\V$-category $\J \hookrightarrow \V$ closed under $\otimes$ and containing the unit object $I$ of $\V$ \cite[3.8, 3.9]{Lu:EnrAlgTh}.  A $\J$-\textit{theory} \cite{Lu:EnrAlgTh} is then defined as a $\V$-category $\T$ whose objects are cotensors $S^J$ of a fixed object $S = S^I$, where $J \in \ob\J \subseteq \ob\V$, the notion of \textit{cotensor} $S^J$ here providing the appropriate concept of `$\V$-enriched $J$-th power' of $S$, written herein as $[J,S]$.  Without loss of generality, we require not only that the objects $[J,S]$ of $\T$ be in bijective correspondence with the objects $J$ of $\J$ but moreover that concretely $\ob\T = \ob\J$.

By considering specific systems of arities $\J \hookrightarrow \V$ one recovers various existing notions as instances of the notion of $\J$-theory, as summarized in the following table; see \cite[\S 3, \S 4.2]{Lu:EnrAlgTh} for details.
\begin{center}
     \begin{tabular}{ | p{39ex} | p{39ex} |}
     \hline
     System of arities $\J \hookrightarrow \V$  & $\J$-theories \\ \hhline{|=|=|}
     $\FinCard \hookrightarrow \Set$, \newline the finite cardinals & Lawvere theories\\ \hline
     $\Vfp \hookrightarrow \V$, \newline the finitely presentable objects, where $\V$ is l.f.p. as a closed category & Power's \textit{enriched Lawvere theories} \cite{Pow:EnrLaw} \\ \hline
     $\J = \V$ & Dubuc's $\V$-\textit{theories} \cite{Dub}; equivalently, arbitrary $\V$-monads on $\V$\\\hline
     $\J = \{I\} \hookrightarrow \V$ & monoids in $\V$ \newline (e.g., rings when $\V = \Ab$)\\ \hline
     all finite copowers of $I$ & the enriched algebraic theories of Borceux and Day \cite{BoDay}\\\hline
     \end{tabular}
\end{center}
Given a $\J$-theory $\T$ and a $\V$-category $\C$, a \textit{$\T$-algebra} in $\C$ is by definition a $\V$-functor $A:\T \rightarrow \C$ that preserves cotensors by objects $J$ of $\J$.  Most often we take $\C = \V$.  We call $\V$-natural transformations between $\T$-algebras \textit{$\T$-homomorphisms}.  Therefore $\T$-algebras in $\C$ form a full sub-$\V$-category $\Alg{\T}_\C \hookrightarrow [\T,\C]$ of the $\V$-category of $\V$-functors from $\T$ to $\C$, provided that the latter $\V$-category exists. But $\Alg{\T}_\C$ may exist even when $[\T,\C]$ does not, and we show herein that $\Alg{\T}_\C$ always exists as soon as $\V$ has equalizers and intersections of $(\ob\J)$-indexed families of strong monomorphisms \pbref{thm:existence_of_vcat_talgs}.

Given a $\J$-theory $\T$, we define for each 4-tuple of objects $J,J',K,K' \in \ob\J = \ob\T$ a pair of morphisms
$$\kk{JJ'KK'},\;\kt{JJ'KK'}\;\;:\;\;\T(J,J')\otimes\T(K,K') \rightarrow \T(J\otimes K,J'\otimes K')$$
which, in the classical case $\J = \FinCard \hookrightarrow \Set = \V$, furnish the first and second Kronecker products of pairs of morphisms in $\T$.  In the general case, we can instead work with pairs of \textit{generalized elements} $\mu:V \rightarrow \T(J,J')$ and $\nu:W \rightarrow \T(K,K')$ of the hom-objects for $\T$, where $V,W \in \ob\V$, and for any such pair we again obtain first and second \textit{Kronecker products} $\mu * \nu, \mu \stt \nu:V \otimes W \rightarrow \T(J \otimes K,J' \otimes K')$.   We say that $\mu$ \textit{commutes with} $\nu$ if the first and second Kronecker products of $\mu$ and $\nu$ are equal, and we say that $\T$ is \textit{commutative} if every such pair $(\mu,\nu)$ commutes, equivalently, if the first and second Kronecker products in $\T$ are equal.

This relation of commutation of generalized elements is symmetric \pbref{thm:cmtn_symm}, and it induces a notion of commutation of \textit{cospans of $\J$-theories}, as follows.  Given $\J$-theories $\T$ and $\U$, a \textit{morphism of $\J$-theories} is an identity-on-objects $\V$-functor $A:\T \rightarrow \U$ satisfying a certain condition \pbref{def:morph_th}.  Given a pair of morphisms of $\J$-theories $A:\T \rightarrow \U$ and $B:\sS \rightarrow \U$, we say that $A$ \textit{commutes with} $B$ if the components $A_{JJ'}:\T(J,J') \rightarrow \U(J,J')$ and $B_{KK'}:\sS(K,K') \rightarrow \U(K,K')$ commute for all $J,J',K,K' \in \ob\J$.  We prove that one can fix $J' = I = K'$ and still obtain an equivalent condition \pbref{thm:comm_mor_th_via_single_outp_ops}.

In order to define a notion of commutant in this general setting, we exploit a connection between commutation and the notion of $\T$-homomorphism \pbref{thm:cmtn_via_thoms}.  Any morphism of $\J$-theories $A:\T \rightarrow \U$ is, in particular, a $\T$-algebra in $\U$ and so can be considered as an object of the $\V$-category $\Alg{\T}_\U$, provided that this $\V$-category exists.  If it does, then for each object $J$ of $\J$ there is a (pointwise) cotensor $[J,A]$ of $A$ by $J$ in $\Alg{\T}_\U$, and we define the \textit{commutant} of $A$ (or of $\T$ with respect to $A$) as the $\J$-theory $\T^\perp_A$ whose hom-objects are the \textit{objects of $\T$-homomorphisms}
$$\T^\perp_A(J,K) = \Alg{\T}_\U([J,A],[K,A]) = [\T,\U]([J,A],[K,A])\;\;\;\;\;\;\;\;(J,K \in \ob\J)$$
with composition and identities as in $\Alg{\T}_\U$.  Hence, as a corollary to our existence result for categories of $\T$-algebras \pbref{thm:existence_of_vcat_talgs}, the commutant $\T^\perp_A$ always exists as soon as $\V$ has equalizers and intersections of $(\ob\J)$-indexed families of strong monomorphisms \pbref{thm:existence_of_commutant_via_intersections}.

The commutant of a morphism of $\J$-theories $A:\T \rightarrow \U$ is a subtheory $\T^\perp_A$ of $\U$ \pbref{thm:cmtnt_is_subth}, and we show that it has a universal property, namely that a morphism of $\J$-theories $B:\sS \rightarrow \U$ commutes with $A$ if and only if $B$ factors through the commutant $\T^\perp_A \hookrightarrow \U$ \pbref{thm:commutants_via_commutativity}.  Letting $\ThJ$ denote the category of $\J$-theories, and calling objects of the slice category $\ThJ \slash \U$ \textit{theories over $\U$}, we show that the assignment to each theory $\T$ over $\U$ its commutant $\T^\perp$ extends to a functor $(-)^\perp:(\ThJ \slash \U)^\op \rightarrow \ThJ \slash \U$ that is right-adjoint to its formal dual \pbref{thm:adjn}.  The resulting adjunction restricts to a Galois connection on subtheories of $\U$ \pbref{thm:galois_connection}.

If a given object $C$ of a $\V$-category $\C$ is equipped with cotensors $[J,C]$ by each object $J$ of $\J \hookrightarrow \V$, then we can form an associated $\J$-theory $\C_C$ called the \textit{full $\J$-theory of $C$ in $\C$} with hom-objects
$$\C_C(J,K) = \C([J,C],[K,C])\;\;\;\;\;\;\;\;(J,K \in \ob\J)$$
and with composition and identities as in $\C$.  In particular, the commutant $\T^\perp_A$ of a morphism of $\J$-theories $A:\T \rightarrow \U$ is the full $\J$-theory
$$\T^\perp_A = (\Alg{\T}_\U)_A$$
of $A$ in $\Alg{\T}_\U$.

This leads to a notion of commutant of an arbitrary $\T$-\textit{algebra}, as follows.  Indeed, since any $\T$-algebra $A:\T \rightarrow \C$ equips its \textit{carrier} $\ca{A} := A(I)$ with cotensors $[J,\ca{A}] = A(J)$ by each object $J$ of $\J$, we can form the full $\J$-theory $\C_{|A|}$ of $\ca{A}$ in $\C$, and then $A$ can be viewed equally as a morphism of $\J$-theories
$$A:\T \rightarrow \C_{|A|}\;.$$
The \textit{commutant} of the $\T$-algebra $A$ is defined as the commutant $\T^\perp_A \hookrightarrow \C_{|A|}$ of this induced morphism.  Consequently, the commutant of $A$ is equivalently characterized as the full $\J$-theory
$$\T^\perp_A = (\Alg{\T}_\C)_A$$
of $A$ in the $\V$-category of $\T$-algebras in $\C$, provided that the latter $\V$-category exists.

Throughout this paper, we refer the reader to numerous examples of commutants for classical Lawvere theories that are developed in detail in the author's recent paper \cite{Lu:CvxAffCmt}.  The present setting of $\V$-enriched $\J$-theories also admits the classical notion of \textit{centralizer} for rings as a source of basic examples, when one takes $\V$ to be the category $\Ab$ of a abelian groups with $\J = \{\ZZ\} \hookrightarrow \Ab$.  For example, given a ring $T$ with corresponding $\{\ZZ\}$-theory $\T$, a $\T$-algebra $M$ is precisely a left $T$-module, and the commutant $\T^\perp_M \hookrightarrow \Ab_M$ of $M$ in the above sense is the inclusion of endomorphism rings $\End_T(M) \hookrightarrow \End_\ZZ(M)$ \pbref{exa:cmtnt_rmod_over_ab}.

Given a system of arities $j:\J \hookrightarrow \V$, we say that a $\V$-monad $\TT = (T,\eta,\mu)$ on $\V$ is a \textit{$\J$-ary monad} if $T$ preserves ($\V$-enriched) left Kan extensions along $j$ \cite[\S 11]{Lu:EnrAlgTh}.  The $\J$-ary monads form a full subcategory $\MndJ(\V) \hookrightarrow \Mnd_{\VCAT}(\V)$ of the category of all $\V$-monads on $\V$, and it is proved in \cite[11.8]{Lu:EnrAlgTh} that there is an equivalence
\begin{equation}\label{eq:equiv_jth_jmnd_intro}\ThJ \;\;\;\;\simeq\;\;\;\; \MndJ(\V)\end{equation}
between the category of $\J$-theories and the category of $\J$-ary monads, provided that the system of arities $\J$ is \textit{eleutheric} (\cite[\S 7]{Lu:EnrAlgTh}, see \bref{par:equiv_jth_jary_mnds} below).  By \cite[7.5]{Lu:EnrAlgTh}, all the systems of arities listed in the above table are always eleutheric save for the last, which is eleutheric for a wide class of categories $\V$ \cite[7.5 \#5]{Lu:EnrAlgTh}.  In particular, by taking $\J = \V$ one obtains an equivalence between $\V$-theories and arbitrary $\V$-monads on an arbitrary symmetric monoidal closed category $\V$.

Whereas Kock defined a notion of commutation of cospans of arbitrary $\V$-monads on $\V$ \cite[4.1]{Kock:DblDln}, we show that the above notion of commutation for cospans of $\J$-theories accords with Kock's notion of commutation, in that a cospan of $\J$-theories commutes if and only if the corresponding cospan of $\J$-ary monads commutes in Kock's sense \pbref{thm:cmmtn_jth_jmnd}.  In particular, a $\J$-theory $\T$ is commutative if and only if its corresponding $\J$-ary monad is commutative \pbref{thm:jary_mnd_comm_iff_jth_comm} in the sense defined by Kock \cite{Kock:Comm}.

Via the equivalence \eqref{eq:equiv_jth_jmnd_intro}, the notion of commutant for $\J$-theories induces a corresponding notion of commutant for $\J$-ary monads \pbref{def:jary_and_abs_cmtnt_for_mnds}.  Indeed, given a morphism of $\J$-ary monads $\alpha:\TT \rightarrow \UU$ we can thus define its \textit{$\J$-ary commutant}, which (if it exists) is a $\J$-ary monad $\TT^\perp_{\alpha,j}$ equipped with a canonical morphism $\TT^\perp_{\alpha,j} \rightarrow \UU$.  In view of the above, the $\J$-ary commutant is characterized by a universal property that we can phrase in terms of Kock's notion of commutation of cospans of monads \pbref{thm:univ_prop_jary_cmmtnt}.  In particular, by taking $\J = \V$ we obtain a notion of commutant for an arbitrary morphism of $\V$-monads $\alpha:\TT \rightarrow \UU$ on $\V$, namely the `$\V$-ary commutant' which we call the \textit{absolute commutant} $\TT^\perp_\alpha$ of $\alpha$ \pbref{def:jary_and_abs_cmtnt_for_mnds}.  We obtain strong general existence results for both $\J$-ary and absolute commutants (\bref{rem:existence_of_cmtnt_for_mnds}).

Given a morphism of $\J$-ary monads $\alpha:\TT \rightarrow \UU$ we can consider both its $\J$-ary commutant $\TT^\perp_{\alpha,j}$ and its absolute commutant $\TT^\perp_\alpha$, each of which is characterized by an (a priori) different universal property when it exists.  As we argue in \bref{rem:jary_vs_absolute}, we have no reason to expect that the $\J$-ary and absolute commutants would coincide in general.  Indeed, whereas the absolute commutant is always a \textit{submonad} $\TT^\perp_\alpha \hookrightarrow \UU$ \pbref{rem:jary_vs_absolute}, we have no reason to expect in general that the canonical morphism $\TT^\perp_{\alpha,j} \rightarrow \UU$ would be componentwise monic \pbref{rem:jary_vs_absolute}.

Nevertheless, we identify one important special case in which the $\J$-ary and absolute commutants coincide, namely the case in which the system of arities is the inclusion $\FinCard \hookrightarrow \Set$.  Indeed, given a morphism of finitary monads $\alpha:\TT \rightarrow \UU$ on $\Set$ we prove that the \textit{finitary commutant} of $\alpha$ coincides with the absolute commutant of $\alpha$ \pbref{thm:abs_cmtnt_is_finitary_cmtnt_over_set}.
  
Given a $\TT$-algebra $A$ for a $\V$-monad $\TT$ on $\V$, we define the \textit{absolute commutant of $A$} as the $\V$-monad $\TT^\perp_A$ corresponding to the commutant $\T^\perp_A$ of the $\T$-algebra $\T \rightarrow \V$ corresponding to $A$, where $\T$ denotes the $\V$-theory associated to $\TT$.  Here the notion of absolute commutant intersects with the notion of \textit{codensity monad} \cite{Kock:CodMnd}, as $\TT^\perp_A$ is equally the codensity $\V$-monad\footnote{See \cite[Ch. II]{Dub} for a definition in the enriched setting.} of the $\V$-functor $\T \rightarrow \V$ in this case \pbref{thm:abs_cmtnt_of_talg_vs_codensity_mnd}.

More generally, for an arbitrary system of arities $\J$ the notion of commutant of a $\T$-algebra $A:\T \rightarrow \C$ intersects with (a $\V$-enriched generalization of) Lawvere's notion of \textit{algebraic structure} \cite[III.1]{Law:PhD} in the case where $\C = \V$ \pbref{rem:jalg_str}.

Beyond our general existence result for commutants \pbref{thm:existence_of_commutant_via_intersections}, we prove that the commutant $\T^\perp_A$ of a $\T$-algebra $A:\T \rightarrow \V$ always exists as soon as $\J \hookrightarrow \V$ is eleutheric and $\V$ has equalizers \pbref{thm:cmtnt_alg_ele}.  In particular, for an arbitrary $\V$-monad $\TT$ on a symmetric monoidal closed category $\V$ with equalizers, the absolute commutant of a $\TT$-algebra $A$ always exists \pbref{def:abs_cmtnt_talg}.

A complementary abstract perspective on notions of commutation in a general framework of duoidal categories is provided by the very recent paper \cite{LfGa}.  The authors define notions of commutation and centralizer in a general setting, but the content, scope, context, methods, aims, and results of the latter article are very different from those of the present paper.  Elements of the present work were announced in a 2015 conference talk \cite{Lu:CT2015}, and the present paper provides part of the basis of a framework for measure and distribution monads outlined in that same talk and expounded in \cite{Lu:FDistn}.

\begin{Acknowledgement}  The author thanks the anonymous referee for helpful suggestions, which enabled a greatly shortened proof of \ref{thm:valued_in_thoms_quantifying_over_just_obj} as well as an improvement to the proof of \ref{thm:obj_of_jop_homs}.
\end{Acknowledgement}

\section{Some basic notions and lemmas}\label{sec:background}

\begin{ParSub}\label{par:str_mono}
A monomorphism $m:C \rightarrow D$ in a category $\C$ is called a \textbf{strong monomorphism} \cite{Ke:MonoEpiPb} provided that for all morphisms $e:A \rightarrow B$, $f:A \rightarrow C$, $g:B \rightarrow D$ in $\C$, if $e$ is an epimorphism and $g \cdot e = m \cdot f$ then $g$ factors through\footnote{Since $m$ is a monomorphism it then follows that the morphism $d:B \rightarrow C$ with $m \cdot d = g$ is unique and satisfies the equation $d \cdot e = f$.} $m$.  A subobject that is represented by a strong monomorphism is said to be a \textbf{strong subobject}.  Given a family of parallel pairs of morphisms $(h_\lambda,k_\lambda:D \rightarrow E_\lambda)_{\lambda \in \Lambda}$ in $\C$ indexed by a class $\Lambda$, let us call a limit of the resulting diagram in $\C$ a \textbf{pairwise equalizer} of the family $(h_\lambda,k_\lambda)$.  Such a limit is equivalently given by a morphism $m:C \rightarrow D$ satisfying an evident universal property, and it is easy to show directly that $m$ is necessarily a strong monomorphism.  If $\C$ has an equalizer $m_\lambda$ for each individual pair $(h_\lambda,k_\lambda)$, then each $m_\lambda$ is necessarily a strong monomorphism \cite[3.1]{Ke:MonoEpiPb}, and a pairwise equalizer of $(h_\lambda,k_\lambda)_{\lambda \in \Lambda}$ is equivalently a \textbf{(wide) intersection} of the family of strong monomorphisms $m_\lambda$, i.e. a fibre product of this family.
\end{ParSub}

\begin{ParSub}
Throughout the sequel, we fix an arbitrary closed symmetric monoidal category $(\V,\otimes,I,a,\ell,r,s)$ and employ the theory of $\V$-enriched categories, as documented in the classic works \cite{EiKe,Dub,Ke:Ba}.  By a \textit{morphism} in a $\V$-category $\C$ we mean a morphism in the ordinary category $\C_0$ underlying $\C$.  Concretely, a morphism $f:C \rightarrow D$ in $\C$ is therefore a morphism $I \rightarrow \C(C,D)$ in $\V$, but nevertheless we sometimes maintain a notational distinction between these notions by writing the latter morphism as $[f]$.  We denote by $\uV$ the $\V$-category canonically associated to $\V$, whose underlying ordinary category may be identified with $\V$ itself.
\end{ParSub}

\begin{ParSub}\label{par:faithful}
Recall that a $\V$-functor $G:\A \rightarrow \X$ is said to be \textbf{faithful} if its component morphisms $G_{AB}:\A(A,B) \rightarrow \X(GA,GB)$ are monomorphisms in $\V$.  We shall say that $G$ is \textbf{strongly faithful} if the $G_{AB}$ are, moreover, strong monomorphisms.
\end{ParSub}

\begin{ParSub}\label{par:cot}
Given an object $C$ of a $\V$-category $\C$ and an object $V$ of $\V$, recall that a \textbf{cotensor} of $C$ by $V$ in $\C$ is, by definition, an object $[V,C]$ of $\C$ equipped with an isomorphism of $\V$-functors
\begin{equation}\label{eq:cot}\C(-,[V,C]) \cong \uV(V,\C(-,C)):\C^\op \rightarrow \uV\;.\end{equation}
Therefore a cotensor of $C$ by $V$ is exactly a \textit{representation} of the rightmost $\V$-functor in \eqref{eq:cot} and so is equivalently given by an object $[V,C]$ with a morphism 
\begin{equation}\label{eq:cot_counit}\gamma_V^C:V \rightarrow \C([V,C],C)\;,\end{equation}
called the \textit{counit} of the representation, having the property that the $\V$-natural transformation $\C(-,[V,C]) \rightarrow \uV(V,\C(-,C))$ determined by $\gamma_V^C$ (via Yoneda) is an isomorphism.

Given a fixed object $V$ of $\V$ and cotensors $[V,C]$ in $\C$ for every object $C$ of some full sub-$\V$-category $\D \hookrightarrow \C$, we deduce by \cite[\S 1.10]{Ke:Ba} that there is a unique $\V$-functor $[V,-]:\D \rightarrow \C$ given on objects by $C \mapsto [V,C]$ such that the counits \eqref{eq:cot_counit} are $\V$-natural in $C \in \D$.  One can of course adapt this in an evident way to the case in which we instead have an arbitrary $\V$-functor $\D \rightarrow \C$ rather than a full sub-$\V$-category inclusion.  In particular, if we are given a pair of objects $(C_1,C_2)$ of $\C$ and cotensors $[V,C_1]$ and $[V,C_2]$ in $\C$ then the mapping $\{1,2\} \rightarrow \ob\C$ given by $i \mapsto C_i$ determines a fully-faithful $\V$-functor $\D \rightarrow \C$ when we define $\D$ to have objects $\{1,2\}$ and homs $\D(i,j) = \C(C_i,C_j)$.  Hence we obtain an induced $\V$-functor $[V,-]:\D \rightarrow \C$.  In particular, the given cotensors $[V,C_i]$ thus induce a morphism $[V,-]_{1,2}:\C(C_1,C_2) \rightarrow \C([V,C_1],[V,C_2])$ that we will sometimes write as $[V,-]_{C_1C_2}$, although strictly speaking this is an abuse of notation.  Indeed, we could even have $C_1 = C_2$ and yet still have a given pair of distinct (but isomorphic) cotensors $[V,C_1]$ and $[V,C_2]$, so that even our way of writing the given pair of cotensors conceals an abuse of notation.  With care in this regard, we will harness the $\V$-functoriality of cotensors in several subtle ways in the sequel by means of the following lemma, which is obvious in the general case but becomes quite useful in the degenerate cases captured by the corollaries that follow it:

\begin{LemSub}\label{thm:cot_lem}
Let $\C$ be a $\V$-category, let $V$ be an object of $\V$, and for each $i = 1,2,3,4$, let $C_i$ be an object of $\C$ equipped with a given cotensor $[V,C_i]$ in $\C$ (noting that the cotensor $[V,C_i]$ depends $i$ rather than just $C_i$).  Let $f_1:C_1 \rightarrow C_2$ and $f_3:C_3 \rightarrow C_4$ be isomorphisms in (the ordinary category underlying) $\C$, and for each $i = 1,3$ write $[V,f_i]:[V,C_i] \rightarrow [V,C_{i + 1}]$ for the induced isomorphism (noting that $[V,f_i]$ depends on $i$ rather than just $f_i$).  Then we have a commutative square
$$
\xymatrix{
\C(C_1,C_3) \ar[rr]^(.4){[V,-]_{13}} \ar[d]_{\C(f_1^{-1},f_3)}^\wr & & \C([V,C_1],[V,C_3]) \ar[d]^{\C([V,f_1]^{-1},[V,f_3])}_\wr \\
\C(C_2,C_4) \ar[rr]_(.4){[V,-]_{24}} & & \C([V,C_2],[V,C_4])
}
$$
whose left and right sides are isomorphisms.
\end{LemSub}
\begin{proof}
The mapping $\{1,2,3,4\} \rightarrow \ob\C$ given by $i \mapsto C_i$ extends to an identity-on-homs $\V$-functor $\D \rightarrow \C$ where $\ob\D = \{1,2,3,4\}$, and  by \cite[\S 1.10]{Ke:Ba} we obtain a $\V$-functor $[V,-]:\D \rightarrow \C$, given on objects by $i \mapsto [V,C_i]$, whose $\V$-functoriality now entails the needed result.
\end{proof}

\begin{CorSub}\label{thm:cots_of_same_obj}
Let $\C$ be a $\V$-category and let $V$ be an object of $\V$.  For each $j = 1,2$, let $D_j$ be an object of $\C$, let $[V,D_j]^{0}$ and $[V,D_j]^{1}$ be a given pair of (possibly distinct) cotensors of $D_j$ by $V$ in $\C$, and let $h_j:[V,D_j]^{0} \rightarrow [V,D_j]^{1}$ denote the induced isomorphism.  Then we have a commutative triangle
$$
\xymatrix{
\C(D_1,D_2) \ar[drr]_(.4){[V,-]^{1}_{12}} \ar[rr]^(.4){[V,-]^{0}_{12}} & & \C([V,D_1]^{0},[V,D_2]^{0}) \ar[d]^{\C(h_1^{-1},h_2)}_\wr\\
& & \C([V,D_1]^{1},[V,D_2]^{1})
}
$$
whose right side is an isomorphism.
\end{CorSub}
\begin{proof}
Invoke \bref{thm:cot_lem} with $C_1 = C_2 = D_1$, $C_3 = C_4 = D_2$, $[V,C_1] = [V,D_1]^{0}$, $[V,C_2] = [V,D_1]^{1}$, $[V,C_3] = [V,D_2]^{0}$, $[V,C_4] = [V,D_2]^{1}$, $f_1 = 1_{D_1}$, and $f_3 = 1_{D_2}$.  Then $h_1 = [V,f_1]$, $h_2 = [V,f_3]$, and the result is obtained.
\end{proof}

\begin{CorSub}\label{thm:cot_structs_on_same_objs}
Let $\C$ be a $\V$-category, let $V$ be an object of $\V$, and for each $i = 1,2,3,4$, let $C_i$ be an object of $\C$.  For each $i = 1,3$, let $f_i:C_i \rightarrow C_{i + 1}$ be an isomorphism in $\C$, and let $E_i$ be an object of $\C$ that is equipped with two cotensor structures
$$[V,C_i] = E_i = [V,C_{i + 1}]$$
in $\C$ such that the induced isomorphism $[V,f_i]:[V,C_i] \rightarrow [V,C_{i + 1}]$ is the identity morphism on $E_i$.  Then we have a commutative triangle
$$
\xymatrix{
\C(C_1,C_3) \ar[drr]^{[V,-]_{13}} \ar[d]_{\C(f_1^{-1},f_3)}^\wr & &  \\
\C(C_2,C_4) \ar[rr]_{[V,-]_{24}} & & \C(E_1,E_3)
}
$$
whose left side is an isomorphism, where $[V,-]_{13}$ and $[V,-]_{24}$ are defined as in \bref{thm:cot_lem}.
\end{CorSub}
\begin{proof}
This follows immediately from \bref{thm:cot_lem}.
\end{proof}
\end{ParSub}

\section{Enriched algebraic theories and their algebras}\label{sec:enr_alg_th,rem:assume_subcat}

In the present section we review some basic material concerning enriched algebraic theories for a system of arities \cite{Lu:EnrAlgTh}, together with certain further points needed for the sequel, and we consider several examples, including a number of specific examples of classical Lawvere theories that are treated in more detail in \cite{Lu:CvxAffCmt}.

\begin{ParSub}[\textbf{$\J$-theories for a system of arities}]\label{par:sys_ar}
In the terminology of \cite{Lu:EnrAlgTh}, a \textbf{system of arities} in a symmetric monoidal closed category $\V$ is a fully faithful symmetric strong monoidal $\V$-functor $j:\J \rightarrowtail \uV$.  Any full sub-$\V$-category $\J \hookrightarrow \uV$ containing $I$ and closed under $\otimes$ is a system of arities, and any system of arities is equivalent to one of this special form \cite[3.8]{Lu:EnrAlgTh}.  Hence by the convention of \cite[3.9]{Lu:EnrAlgTh} we often write as if given systems of arities are of this form, and for many purposes we can assume this without loss of generality.  Given a system of arities $j:\J \hookrightarrow \uV$, a $\J$-\textbf{theory} \cite[4.1]{Lu:EnrAlgTh} is a $\V$-category $\T$ equipped with an identity-on-objects $\V$-functor $\tau:\J^\op \rightarrow \T$ that preserves $\J$-cotensors, i.e. that preserves cotensors by all objects $J$ of $\J$ (or, rather, their associated objects $j(J)$ of $\V$).  The notion of $\J$-theory specializes to yield various different existing notions of enriched algebraic theory for different choices of $\J$ and $\V$, as in the following examples from \cite[\S 3, 4.2]{Lu:EnrAlgTh}:
\end{ParSub}

\begin{ExaSub}\label{exa:sys_ar_jth}\emptybox
\begin{enumerate}
\item[(a)] Letting $\V = \Set$, we can take $\J = \FinCard \hookrightarrow \V$ to be the full subcategory consisting of all finite cardinals, and then the resulting notion of $\J$-theory is Lawvere's notion of algebraic theory \cite{Law:PhD}.  These are often called \textbf{Lawvere theories}.
\item[(b)] Letting $\V$ be \textit{locally finitely presentable as a closed category} \cite{Ke:FL}, we can take $\J \hookrightarrow \uV$ to be the full sub-$\V$-category $\Vfp$ consisting of the finitely presentable objects, and then the resulting notion of $\J$-theory is the notion of \textit{enriched Lawvere theory} defined by Power in \cite{Pow:EnrLaw}.
\item[(c)] Letting $\J = \uV$ and taking $j:\uV \rightarrow \uV$ to be the identity $\V$-functor, the resulting notion of $\J$-theory is Dubuc's notion of \textit{$\V$-theory} \cite{Dub:EnrStrSem}, which coincides, up to an equivalence, with the notion of $\V$-monad on $\uV$ \cite[11.10]{Lu:EnrAlgTh}.
\item[(d)] The one-object full sub-$\V$-category $\{I\} \hookrightarrow \uV$ carries the structure of a system of arities, and $\{I\}$-theories are the same as monoids in the monoidal category $\V$.  This example is analyzed in \cite[3.6, 4.2]{Lu:EnrAlgTh} on the basis of the fact that the $\V$-category $\{I\}$ is isomorphic to the \textit{unit $\V$-category} $\mathbb{I}$, which is the one-object $\V$-category determined by the commutative monoid $I$ in $\V$ and has the property that $\V$-functors $\mathbb{I} \rightarrow \C$ valued in any $\V$-category $\C$ correspond bijectively to objects of $\C$.  When $\V = \Ab$ is the category of abelian groups, $\{\ZZ\}$-theories are the same as rings.
\item[(e)] Assuming that $\V$ has finite copowers $n \cdot I$ $(n \in \NN)$ of the unit object $I$, there is a system of arities $j:\NN_\V \rightarrowtail \uV$ with $\ob\NN_\V = \NN$, such that $j$ is given on objects by $n \mapsto n \cdot I$ and $j$ is identity-on-homs.  $\NN_\V$ is a symmetric strict monoidal $\V$-category under multiplication of natural numbers.  The resulting notion of $\J$-theory for this particular system of arities is equivalent to the notion of enriched algebraic theory defined by Borceux and Day in \cite{BoDay}; see \cite[4.2 \#6]{Lu:EnrAlgTh}.
\end{enumerate}
\end{ExaSub}

\begin{ParSub}\label{par:des_jcot}
An object $C$ of a $\V$-category $\C$ is said to have \textbf{designated} $\J$-\textbf{cotensors} if it is equipped with a specified choice of cotensor $[J,C]$ in $\C$ for each object $J$ of $\J$.  These designated $\J$-cotensors are said to be \textbf{standard} if $[I,C]$ is just $C$ itself, with the identity morphism $I \rightarrow \C(C,C)$ as counit.  We say that $\C$ has \textbf{(standard) designated} $\J$-\textbf{cotensors} if each object of $\C$ has (standard) designated $\J$-cotensors.  For example, $\uV$ itself is endowed with standard designated $\J$-cotensors $[J,V]$ of each of its objects $V$ when we force $[I,V] = V$ and take $[J,V] = \uV(J,V)$ otherwise.  In any $\J$-theory $\T$ the object $I$ has standard designated $\J$-cotensors (by \cite[4.3]{Lu:EnrAlgTh}) since each object $J$ of $\J$ serves as a cotensor $[J,I] = J$ in $\T$, with counit $\gamma_J$ defined as the composite
\begin{equation}\label{eq:cot_counit_gammaj}J \xrightarrow{\sim} \uV(I,J) = \J^\op(J,I) \xrightarrow{\tau_{JI}} \T(J,I)\;,\end{equation}
whose first factor is the canonical isomorphism.  Moreover, every $\J$-theory $\T$ has standard designated $\J$-cotensors of each of its objects; see \bref{par:lrdes_cots} below.  In fact, by \cite[5.8]{Lu:EnrAlgTh} the notion of $\J$-theory is equivalently defined as a $\V$-category $\T$ with $\ob\T = \ob\J$ in which each object $J$ is equipped with the structure of a cotensor $[J,I]$ such that these designated $\J$-cotensors are standard.  In this way, the seemingly trifling condition of standardness of $\J$-cotensors is in fact implicit in the definition of $\J$-theory.
\end{ParSub}

\begin{ParSub}[$\T$-\textbf{algebras}]\label{par:talgs} Given a $\J$-theory $\T$ and a $\V$-category $\C$, a $\T$-\textbf{algebra} in $\C$ is a $\J$-cotensor-preserving $\V$-functor $A:\T \rightarrow \C$.  We often call $\T$-algebras in $\uV$ simply $\T$-algebras.  A $\V$-functor $A:\T \rightarrow \C$ is a $\T$-algebra as soon as it preserves $\J$-cotensors of $I$ (\cite[5.9]{Lu:EnrAlgTh}).  Given a $\T$-algebra $A:\T \rightarrow \C$, we call the object $\ca{A} := AI$ of $\C$ the \textbf{carrier} of $A$.  Since $\T$ has standard designated $\J$-cotensors $[J,I] = J$ of $I$ and $A$ preserves $\J$-cotensors, $AJ$ is a cotensor of $\ca{A}$ by $J$ for each object $J$ of $\J$, and $A$ thus equips its carrier $\ca{A}$ with standard designated $\J$-cotensors.  Now supposing that $\C$ (already) has standard designated $\J$-cotensors, a \textbf{normal} $\T$-\textbf{algebra} in $\C$ is, by definition, a $\V$-functor $A:\T \rightarrow \C$ that \textit{strictly} preserves the designated $\J$-cotensors $[J,I] = J$ of $I$ in $\T$, i.e. sends them to the designated $\J$-cotensors $[J,\ca{A}]$ of $\ca{A}$ in $\C$.
\end{ParSub}

\begin{ParSub}[\textbf{The $\V$-category of} $\T$-\textbf{algebras}]\label{par:vcat_talgs}
Given $\T$-algebras $A,B:\T \rightarrow \C$, we call $\V$-natural transformations between $\T$-algebras $\T$-\textbf{homomorphisms}.  The \textbf{object of} $\T$-\textbf{homo}\-\textbf{mor}\-\textbf{phisms} from $A$ to $B$ is, by definition, the object of $\V$-natural transformations from $A$ to $B$, i.e. the end
$$\Alg{\T}_\C(A,B) = \int_{J \in \T}\C(AJ,BJ)$$
in $\uV$, which may or may not exist.  If $\Alg{\T}_\C(A,B)$ exists for all $\T$-algebras $A$ and $B$ in $\C$, then we obtain a $\V$-category $\Alg{\T}_\C$ whose objects are the $\T$-algebras in $\C$.    Analogously we define the $\V$-category $\Alg{\T}^!_\C$ of normal $\T$-algebras in $\C$, which is then a full sub-$\V$-category of $\Alg{\T}_\C$ when the latter exists.  In this case, there is in fact an equivalence of $\V$-categories $\Alg{\T}^!_\C \simeq \Alg{\T}_\C$ (\cite[5.14]{Lu:EnrAlgTh}).  In the case where $\C = \uV$, we often write simply $\Alg{\T}$ (resp. $\Alg{\T}^!$) for the $\V$-category of $\T$-algebras in $\uV$.  

When $\Alg{\T}_\C$ exists, we obtain by \cite[\S 2.2]{Ke:Ba} a $\V$-functor
$$\ca{\text{$-$}} = \Ev_I:\Alg{\T}_\C \rightarrow \C$$
given by evaluation at $I$.  Therefore $\ca{\text{$-$}}$ sends each $\T$-algebra $A$ to its carrier $\ca{A}$.
\end{ParSub}

\begin{ExaSub}[\textbf{Left $R$-modules with} $\V = \Ab$]
For the system of arities $\{I\} \hookrightarrow \uV$ of \bref{exa:sys_ar_jth}(d),   we know that an $\{I\}$-theory $\R$ is the same as a monoid $R$ in $\V$, with $R = \R(I,I)$, and in the case of $\V =\Ab$ (where $I = \ZZ$) these are rings.  Moreover, an $\R$-algebra $M:\R \rightarrow \uV$ in the above sense is the same as a left $R$-module $M$ in $\V$ \cite[5.3 \#3]{Lu:EnrAlgTh}.  For example, when $\V = \Ab$ these are precisely left $R$-modules in the usual sense.  Thus a ring $R$ can be viewed as the $\{\ZZ\}$-theory of left $R$-modules.
\end{ExaSub}

\begin{ExaSub}[\textbf{Left $R$-modules with} $\V = \Set$]\label{exa:law_th_rmods}
The category $\Mod{R}$ of left $R$-modules for a ring $R$, or more generally a \textit{rig} (or \textit{semiring}) $R$, is isomorphic to the category of normal $\T$-algebras $\Alg{\T}^!$ for a Lawvere theory $\T$; see, e.g., \cite[2.8]{Lu:CvxAffCmt}.  The associated theory $\T$ is the category $\Mat_R$ of $R$-\textbf{matrices}, whose objects are natural numbers and whose morphisms $X:n \rightarrow m$ are $m \times n$-matrices with entries in $R$, with composition given by matrix multiplication.
\end{ExaSub}

\begin{ExaSub}[\textbf{The Lawvere theory of commutative $k$-algebras}]\label{exa:lth_ckalgs}
Given a commutative ring $k$, the category of commutative $k$-algebras is isomorphic to the category $\Alg{\T}^!$ of normal $\T$-algebras for a Lawvere theory $\T$ in which $\T(n,1) = k[x_1,...,x_n]$ is the set of polynomials in $n$ variables over $k$; see, e.g. \cite[2.9]{Lu:CvxAffCmt}
\end{ExaSub}

\begin{ExaSub}[\textbf{The Lawvere theory of semilattices}]\label{exa:slats}
A \textbf{(bounded) join semilattice} is a partially ordered set with finite joins.  Equipping the set $2 = \{0,1\}$ with the structure of a rig with additive monoid $(2,\vee,0)$ and multiplicative monoid $(2,\wedge,1)$, the category of join semilattices (and maps preserving finite joins) is isomorphic to the category $\Mod{2}$ of $2$-modules and so (by \bref{exa:law_th_rmods}) is isomorphic to the category $\Alg{\T}^!$ of normal $\T$-algebras for the Lawvere theory $\T = \Mat_2$.  See, e.g., \cite[2.10]{Lu:CvxAffCmt}.
\end{ExaSub}

\begin{ParSub}[\textbf{Morphisms of} $\J$-\textbf{theories}]\label{def:morph_th}
Given $\J$-theories $(\T,\tau)$ and $(\U,\upsilon)$, a \textbf{(normal) morphism of} $\J$-\textbf{theories} $A:\T \rightarrow \U$ is a $\V$-functor such that $A \circ \tau = \upsilon$.  A morphism of $\J$-theories $A:\T \rightarrow \U$ is the same as a normal $\T$-algebra in $\U$ with carrier $I$ \cite[5.16]{Lu:EnrAlgTh}.  Observe that $\J^\op$, when equipped with the identity $\V$-functor, is an initial object of the resulting \textbf{category of} $\J$-\textbf{theories} $\ThJ$.  A \textbf{subtheory} of a $\J$-theory $\U$ is a $\J$-theory $\T$ equipped with a morphism $\iota:\T \hookrightarrow \U$ that is faithful (as a $\V$-functor, \bref{par:faithful}), and we say that $\T$ is a \textbf{strong subtheory} of $\U$ if, moreover, $\iota$ is strongly faithful \pbref{par:faithful}.
\end{ParSub}

\begin{ExaSub}[\textbf{Ring homomorphisms}]
Given monoids $R$ and $U$ in $\V$, we can consider $R$ and $U$ as $\{I\}$-theories $\R$ and $\U$ for the system of arities $\{I\} \hookrightarrow \uV$, and then morphisms of $\{I\}$-theories $a:\R \rightarrow \U$ are the same as homomorphisms of monoids $a:R \rightarrow U$ in $\V$.  When $\V = \Ab$, these are the same as ring homomorphisms $R \rightarrow U$.
\end{ExaSub}

\begin{RemSub}\label{rem:morphs_uniq_det_ji_comps}
Given a normal $\T$-algebra $A:\T \rightarrow \C$ with carrier $C$ (and in particular, any morphism of $\J$-theories), $A$ preserves the designated cotensors $[J,I] = J$ of $I$ and so it follows that for all $J,K \in \ob\J$ we have a commutative square
$$
\xymatrix{
\T(J,K) \ar[d]_\wr \ar[rr]^{A_{JK}} & & \C([J,\ca{A}],[K,\ca{A}]) \ar[d]^\wr\\
\uV(K,\T(J,I)) \ar[rr]_{\uV(K,A_{JI})} & & \uV(K,\C([J,\ca{A}],\ca{A}))
}
$$
whose left and right sides are isomorphisms.  Thus $A_{JK}$ can be expressed in terms of $A_{JI}$.  Hence a normal $\T$-algebra $A$ is uniquely determined by its carrier and its components $A_{JI}$ $(J \in \ob\J)$.
\end{RemSub}

\begin{ExaSub}[\textbf{Affine spaces over a ring or rig}]\label{exa:raff}
Let $R$ be a ring, or more generally, a rig.  Recall that the category of $R$-matrices $\Mat_R$ is the Lawvere theory of left $R$-modules \pbref{exa:law_th_rmods}.  There is a subtheory $\Mat_R^\aff$ of $\Mat_R$ consisting of those matrices in which each row sums to $1$, and we call normal $\Mat_R^\aff$-algebras \textbf{(left)} $R$-\textbf{affine spaces}.  See, e.g., \cite[3.2]{Lu:CvxAffCmt}.
\end{ExaSub}

\begin{ExaSub}[\textbf{Convex spaces}]\label{exa:cvx_sp}
The set $\RR_+$ of all non-negative real numbers is a rig, as it is a subrig of the ring $\RR$.  We call $\RR_+$-affine spaces ($\RR$-)\textbf{convex spaces}.  See, e.g., \cite{Lu:CvxAffCmt}.
\end{ExaSub}

\begin{ExaSub}[\textbf{Unbounded semilattices as affine spaces}]\label{exa:unb_slat}
An \textbf{unbounded join semilattice} is a poset in which every pair of elements has a join.  The category of unbounded join semilattices and maps preserving binary joins is isomorphic to the category of affine spaces over the rig $(2,\vee,0,\wedge,1)$ \cite[3.3]{Lu:CvxAffCmt}.
\end{ExaSub}

The following is a direct generalization of \S 2.11 of the author's paper \cite{Lu:CvxAffCmt} in the finitary $\Set$-based case, which we have adapted word-for-word in order to clearly emphasize the parallel:

\begin{DefSub}[\textbf{The full theory of an object}]\label{def:full_theory}
If a given object $C$ of a $\V$-category $\C$ has standard designated $\J$-cotensors $[J,C]$ then we obtain a $\J$-theory $\C_C$, called the \textbf{full $\J$-theory of $C$} in $\C$, with
$$\C_C(J,K) = \C([J,C],[K,C]),\;\;\;\;\;\;J,K \in \ob\C_C = \ob\J$$
such that the mapping $\ob\J \rightarrow \ob\C$, $J \mapsto [J,C]$, extends to an identity-on-homs $\V$-functor $\C_C \rightarrowtail \C$, which is evidently a $\C_C$-algebra in $\C$ with carrier $C$.  When $\V = \Set$ and $\J = \FinCard$, we call $\C_C$ the \textit{full finitary theory} of $C$ in $\C$.

In particular, any $\T$-algebra $A:\T \rightarrow \C$ endows its carrier $\ca{A} = AI$ with standard designated $\J$-cotensors $[J,\ca{A}] = AJ$ \pbref{par:talgs}, with respect to which we can form the full $\J$-theory of $\ca{A}$, which we shall denote by $\C_A$.  The given $\T$-algebra $A$ then factors uniquely as
$$
\xymatrix{
\T \ar[dr]_A \ar@{..>}[r]^{A'} & \C_A \ar@{ >->}[d]\\
                         & \C
}
$$
where $A'$ is a morphism of $\J$-theories, given on homs just as $A$.  By abuse of notation, we often write simply $A$ to denote the morphism $A'$.

In the case that $\C$ has standard designated $\J$-cotensors, morphisms of $\J$-theories $\T \rightarrow \C_C$ into the full $\J$-theory of an object $C$ of $\C$ are evidently in bijective correspondence with normal $\T$-algebras in $\C$ with carrier $C$.  Note also that the canonical $\C_C$-algebra $\C_C \rightarrowtail \C$ is normal in this case.
\end{DefSub}

\begin{ExaSub}[\textbf{The endomorphism ring of an abelian group}]
Take $\V = \Ab$ and $\J = \{\ZZ\}$, and let $M$ be an abelian group.  Then the full $\{\ZZ\}$-theory $\Ab_M$ of $M$ in $\Ab$ is the ring $\End_\ZZ(M)$ of all endomorphisms of $M$.
\end{ExaSub}

\begin{ExaSub}[\textbf{The Lawvere theory of Boolean algebras}]\label{exa:bool}
The category of Boolean algebras is isomorphic to the category $\Alg{\T}^!$ of normal $\T$-algebras where $\T = \Set_2$ is the full finitary theory of $2 = \{0,1\}$ in $\Set$; see \cite[III.1, Example 4]{Law:PhD} and \cite[2.12]{Lu:CvxAffCmt}.
\end{ExaSub}

\begin{ParSub}[\textbf{The left, right, and designated $\J$-cotensors in a theory $\T$}]\label{par:lrdes_cots}
As we noted above, every $\J$-theory $\T$ has all $\J$-cotensors, and in the sequel it will be convenient to make use of multiple distinct ways of forming $\J$-cotensors in $\T$, with separate notations for each, as follows.

\begin{enumerate}
\item Firstly, for each pair of objects $J,K$ of $\J$, the coevaluation morphism
$$\Coev:J \rightarrow \uV(K,J\otimes K) = \J^\op(J\otimes K,K)$$
exhibits $J\otimes K$ as a cotensor $[J,K]$ of $K$ by $J$ in $\J^\op$.  Hence the composite
$$\gamma^K_J = \left(J \xrightarrow{\Coev} \J^\op(J\otimes K,K) \xrightarrow{\tau_{J\otimes K,K}} \T(J\otimes K,K)\right)$$
exhibits $J \otimes K$ as a cotensor $[J,K]$ in $\T$, which we write as 
$$[J,K]_\ell = J \otimes K$$
and call the \textbf{left cotensor} of $K$ by $J$.

\item Secondly, since $\V$ is \textit{symmetric} monoidal closed, we have another coevaluation morphism $\Coev':J \rightarrow \uV(K,K\otimes J)$ that is related to the morphism $\Coev$ from 1 via the equation $\Coev' = \uV(K,s_{JK}) \cdot \Coev$, where $s_{JK}:J \otimes K \rightarrow K \otimes J$ is the symmetry.  It follows that the composites
\begin{equation}\label{eq:direct_charn_right_cot}J \xrightarrow{\Coev'} \uV(K,K\otimes J) = \J^\op(K\otimes J,K) \xrightarrow{\tau_{K\otimes J,K}} \T(K\otimes J,K)\end{equation}
and
$$J \xrightarrow{\gamma^K_J} \T(J\otimes K,K) \xrightarrow{\T(\tau(s_{JK}),1)} \T(K\otimes J,K)$$
are equal and present $K \otimes J$ as a cotensor of $K$ by $J$ in $\T$, which we write as 
$$[J,K]_r = K\otimes J$$
and call the \textbf{right cotensor} of $K$ by $J$.

\item Thirdly, recall that the objects $J$ of $\T$ themselves serve as standard designated $\J$-cotensors $[J,I] = J$ of $I$ \pbref{par:des_jcot}.  These are in general neither the left nor the right cotensors (which are not standard in general), and so it is convenient to fix a choice of standard designated $\J$-cotensors $[J,K]$ in $\T$ that coincides with the basic choice $[J,I] = J$ in the case that $K = I$.  We shall call these the \textbf{(standard) designated} $\J$-\textbf{cotensors} in $\T$ and write them simply as $[J,K]$.
\end{enumerate}
\end{ParSub}

\begin{RemSub}\label{rem:morph_pres_lr_jcots}
A morphism of $\J$-theories $A:\T \rightarrow \U$ strictly preserves the left $\J$-cotensors $[J,K]_\ell = J\otimes K$ and also the right $\J$-cotensors $[J,K]_r = K\otimes J$. Indeed, this follows immediately from the descriptions of the right and left cotensor counits given in \bref{par:lrdes_cots}.
\end{RemSub}

\begin{ParSub}[\textbf{Cotensors of algebras}]\label{par:cot_alg}
If $\C$ is a $\V$-category with $\J$-cotensors, then the $\V$-category of $\T$-algebras $\Alg{\T}_\C$ has $\J$-cotensors as soon as it exists.  Indeed, given an object $J$ of $\J$ and a $\T$-algebra $A:\T \rightarrow \C$, a cotensor $[J,A]$ can be formed \textit{pointwise}, as the composite
$$[J,A-] = \left(\T \xrightarrow{A} \C \xrightarrow{[J,-]} \C\right),$$
which is a $\T$-algebra since $[J,-]$ preserves cotensors.

In the case where $\C$ is itself a $\J$-theory $\C = \U$, the three canonical choices of $\J$-cotensors in $\U$ \pbref{par:lrdes_cots} give rise to three different choices of pointwise $\J$-cotensors in $\Alg{\T}_\U$, namely the \textit{(pointwise) left cotensors} $[J,A]_\ell = [J,A-]_\ell$, the \textit{(pointwise) right cotensors} $[J,A]_r = [J,A-]_r$, and the \textit{(pointwise) designated cotensors} $[J,A] = [J,A-]$. 

For a morphism of $\J$-theories $A:\T \rightarrow \U$, we have
\begin{equation}\label{eq:mor_com_lcot}[J,A]_\ell = [J,A-]_\ell = [J,-]_\ell \circ A = A \circ [J,-]_\ell = A([J,-]_\ell)\end{equation}
\begin{equation}\label{eq:mor_com_rcot}[J,A]_r = [J,A-]_r = [J,-]_r \circ A = A \circ [J,-]_r = A([J,-]_r)\end{equation}
for all $J \in \ob\J$, since $A$ strictly preserves the right and left $\J$-cotensors \pbref{rem:morph_pres_lr_jcots}. 
\end{ParSub}

\section{The object of homomorphisms}

Our study of commutation and commutants for $\J$-theories will be enabled by a detailed study of the \textit{object of $\T$-homomorphisms} $\Alg{\T}_\C(A,B) = \int_{J \in \T}\C(AJ,BJ)$ for a pair of $\T$-algebras $A,B:\T \rightarrow \C$ \pbref{par:vcat_talgs}.  We begin by treating the case of the initial $\J$-theory $\J^\op$.

\begin{PropSub}\label{thm:obj_of_jop_homs}
For all $\J^\op$-algebras $A,B:\J^\op \rightarrow \C$, there are morphisms
\begin{equation}\label{eq:obj_jop_homs}\C(\ca{A},\ca{B}) \xrightarrow{\lambda^{AB}_J\:=\:[J,-]_{|A||B|}} \C(AJ,BJ)\;\;\;\;(J \in \J)\end{equation}
that present $\C(\ca{A},\ca{B})$  as the object of $\J^\op$-homomorphisms
$$\C(\ca{A},\ca{B}) = [\J^\op,\C](A,B)\;.$$
\end{PropSub}
\begin{proof}
By \cite[5.8]{Lu:EnrAlgTh}, $B$ is the $\V$-functor $[-,BI]:\J^\op \rightarrow \C$ induced by the cotensors $BJ = [J,BI]$ $(J \in \ob\J)$.  Hence if we let $\mathbb{I}$ denote the unit $\V$-category and let $\iota:\mathbb{I} \rightarrow \J^\op$ denote the $\V$-functor determined by the object $I$ of $\J^\op$, then the identity transformation $B\iota \Rightarrow B\iota$ presents $B$ as a right Kan extension of $B\iota$ along $\iota$.  By \cite[Thm. 4.38]{Ke:Ba} we therefore have an isomorphism
$$\varphi\;:\;[\J^\op,\C](A,B) \xrightarrow{\sim} [\mathbb{I},\C](A\iota,B\iota) = \C(AI,BI),$$
given by evaluation at $I$, and in particular, the object of $\V$-natural transformations $[\J^\op,\C](A,B)$ exists in $\V$.

The cotensors $AJ = [J,AI]$ and $BJ = [J,BI]$ $(J \in \ob\J)$ induce an extraordinarily $\V$-natural family as in \eqref{eq:obj_jop_homs}, and the induced morphism $\lambda^\sharp:\C(AI,BI) \rightarrow [\J^\op,\C](A,B)$ is a section of $\varphi$, so $\lambda^\sharp = \varphi^{-1}$.
\end{proof}

\begin{CorSub}
The $\V$-category $\Alg{\J^\op}_\C$ always exists.  If $\C$ has standard designated $\J$-cotensors, then $\Alg{\J^\op}^!_\C$ is isomorphic to $\C$, which is therefore equivalent to $\Alg{\J^\op}_\C$. 
\end{CorSub}
\begin{proof}
It follows immediately from \cite[5.7]{Lu:EnrAlgTh} that the assignment to each normal $\J^\op$-algebra $A$ its carrier $\ca{A}$ is a bijection between normal $\J^\op$-algebras and objects of $\C$.  Hence the result follows from the preceding Proposition.
\end{proof}

\begin{RemSub}
Despite the equivalence $\Alg{\J^\op}_\C \simeq \C$, there is a useful distinction to be made between $\J^\op$-algebras and mere objects of $\C$.  Indeed, by \cite[5.7]{Lu:EnrAlgTh}, a $\J^\op$-algebra is precisely an object $C$ of $\C$ together with a choice of standard designated $\J$-cotensors $[J,C]$ $(J \in \ob\J)$.

In particular, every $\T$-algebra $A:\T \rightarrow \C$ comes equipped with a choice of $\J$-cotensors for its carrier $\ca{A}$, and this information is encapsulated by the associated $\J^\op$-algebra
$$A \circ \tau = \left(\J^\op \xrightarrow{\tau} \T \xrightarrow{A} \C\right).$$
\end{RemSub}

\begin{DefSub}\label{def:fam_t_homs}
Let $A,B:\T \rightarrow \C$ be $\T$-algebras, and let $f:V \rightarrow \C(\ca{A},\ca{B})$ be a morphism in $\V$.  By \bref{thm:obj_of_jop_homs}, $\C(\ca{A},\ca{B})$ is an end $\int_{J \in \J^\op}\C((A \circ \tau)J,(B \circ \tau)J)$, so $f$ determines a corresponding family of morphisms
$$f_J:V \rightarrow \C((A\circ\tau)J,(B\circ\tau)J) = \C(AJ,BJ)\;,$$
$\V$-natural in $J \in \J^\op$.  We say that $f$ \textbf{is valued in} $\T$-\textbf{homomorphisms} from $A$ to $B$ if the latter family is $\V$-natural in $J \in \T$, i.e. if $(f_J:V \rightarrow \C(AJ,BJ))$ is an extraordinarily $\V$-natural family for the $\V$-functor
$$\C(A-,B-):\T^\op \otimes \T \rightarrow \uV\;.$$
In the special case where $V = I$, we say that a morphism $f:\ca{A} \rightarrow \ca{B}$ in $\C_0$ is a $\T$\nolinebreak\mbox{-}\nolinebreak\textbf{homomorphism} from $A$ to $B$ if $f:I \rightarrow \C(\ca{A},\ca{B})$ is valued in $\T$-homomorphisms, equivalently, if the corresponding family $(f_J:I \rightarrow \C(AJ,BJ))$ is a $\T$-homomorphism $A \Rightarrow B$ in the sense of \bref{par:vcat_talgs}.
\end{DefSub}

\begin{ParSub}\label{par:charns_t_homs}
Let $A,B$ and $f:V \rightarrow \C(\ca{A},\ca{B})$ be as in the preceding Definition.  By definition, $f$ is valued in $\T$-homomorphisms iff the diagram
\begin{equation}\label{eq:diag_fam_thoms}
\xymatrix{
\T(J,K) \ar[d]_{\C(A-,BK)_{KJ}} \ar[r]^(.4){\C(AJ,B-)_{JK}} & \uV(\C(AJ,BJ),\C(AJ,BK)) \ar[d]^{\uV(f_J,1)}\\
\uV(\C(AK,BK),\C(AJ,BK)) \ar[r]_(.6){\uV(f_K,1)} & \uV(V,\C(AJ,BK))
}
\end{equation}
commutes for all $J,K \in \ob\T = \ob\J$.  Defining
$$\phi_{JK} := \left(\T(J,K)\otimes\C(\ca{A},\ca{B}) \xrightarrow{A_{JK}\otimes\lambda^{AB}_K} \C(AJ,AK)\otimes\C(AK,BK) \xrightarrow{c} \C(AJ,BK)\right),$$
$$\psi_{JK} := \left(\T(J,K)\otimes\C(\ca{A},\ca{B})\xrightarrow{B_{JK}\otimes\lambda^{AB}_J} \C(BJ,BK)\otimes\C(AJ,BJ) \xrightarrow{c} \C(AJ,BK)\right),$$
where $c$ denotes the relevant composition morphism, we find that $f$ is valued in $\T$-homomorphisms if and only if the diagram
\begin{equation}\label{eq:diag_fam_thoms_ii}
\xymatrix{
\T(J,K)\otimes V \ar[d]_{1\otimes f} \ar[r]^(.4){1\otimes f} & \T(J,K)\otimes\C(\ca{A},\ca{B}) \ar[d]^{\psi_{JK}}\\
\T(J,K)\otimes\C(\ca{A},\ca{B}) \ar[r]_(.6){\phi_{JK}} & \C(AJ,BK)
}
\end{equation}
commutes for every pair of objects $J,K \in \ob\J$, since the two composites in this diagram are exactly the transposes of the two composites in \eqref{eq:diag_fam_thoms}.  Transposing once again, we therefore obtain the following: 
\end{ParSub}

\begin{PropSub}\label{thm:lem_charn_fam_thoms}
Let $f:V \rightarrow \C(\ca{A},\ca{B})$ be as in \bref{def:fam_t_homs}, and for all $J,K \in \ob\J$, let
\begin{equation}\label{eq:pairs_specifying_obj_thoms}\Phi_{JK},\Psi_{JK}:\C(\ca{A},\ca{B}) \rightarrow \uV(\T(J,K),\C(AJ,BK))\end{equation}
be the transposes of the morphisms $\phi_{JK},\psi_{JK}$ of \bref{par:charns_t_homs}.  Then $f$ is valued in $\T$-homo\-morphisms iff the following equations hold:
\begin{equation}\label{eq:thoms_eqns}\Phi_{JK} \cdot f = \Psi_{JK} \cdot f\;\;\;\;\;\;\;\;(J,K \in \ob\J).\end{equation}
\end{PropSub}

It now follows that the object of $\T$-homomorphisms can be equivalently characterized as a certain pairwise equalizer \pbref{par:str_mono} in $\V$, as follows:

\begin{ThmSub}\label{thm:existence_of_vcat_of_talgs}\emptybox
\begin{enumerate}
\item If $\V$ has equalizers and wide intersections of arbitrary (class-indexed) families of strong subobjects, then the $\V$-category of $\T$-algebras $\Alg{\T}_\C$ exists for every $\V$-category $\C$.
\item Given $\T$-algebras $A,B:\T \rightarrow \C$, the object of $\T$-homomorphisms from $A$ to $B$ is equivalently defined as a pairwise equalizer of the family of parallel pairs \eqref{eq:pairs_specifying_obj_thoms}, i.e., a strong subobject
$$\Alg{\T}_\C(A,B) \hookrightarrow \C(\ca{A},\ca{B})$$
characterized by the property that an arbitrary morphism $f:V \rightarrow \C(\ca{A},\ca{B})$ factors through this subobject iff $f$ is valued in $\T$-homomorphisms \pbref{def:fam_t_homs}.
\end{enumerate}
\end{ThmSub}
\begin{proof}
Let us prove 2, as 1 then follows by the remarks in \bref{par:str_mono}.  By \bref{def:fam_t_homs} and \bref{thm:lem_charn_fam_thoms}, $\V$-natural families $f_J:V \rightarrow \C(AJ,BJ)$ $(J \in \T)$ are in bijective correspondence with morphisms $f:V \rightarrow \C(\ca{A},\ca{B})$ that satisfy the equations \eqref{eq:thoms_eqns}, where $f = f_I$ under this bijection.  With reference to the definition of ends in $\uV$ given in \cite[\S 2.1]{Ke:Ba}, the result follows.
\end{proof}

\begin{RemSub}\label{rem:carrier_functor_str_ff}
When $\Alg{\T}_\C$ exists, the $\V$-functor $\ca{\text{$-$}} = \Ev_I:\Alg{\T}_\C \rightarrow \C$ is strongly faithful \pbref{par:faithful} since its structure morphisms are exactly the strong monomorphisms $\ca{\text{$-$}}_{AB}:\Alg{\T}_\C(A,B) \hookrightarrow \C(\ca{A},\ca{B})$ of \bref{thm:existence_of_vcat_of_talgs}.
\end{RemSub}

It will be convenient to introduce the following terminology for the sequel:

\begin{DefSub}\label{def:pres_ops}
Given $\T$-algebras $A,B:\T \rightarrow \C$, a morphism $f:V \rightarrow \T(\ca{A},\ca{B})$ in $\V$, and objects $J,K$ of $\J$, we say that $f$ \textbf{preserves} $\T$-\textbf{operations of input arity $J$ and output arity $K$} if the following equivalent conditions are satisfied: (i) Equation \eqref{eq:thoms_eqns} holds; (ii) the diagram \eqref{eq:diag_fam_thoms_ii} commutes; (iii) the diagram \eqref{eq:diag_fam_thoms} commutes.
\end{DefSub}

Note that $f$ is valued in $\T$-homomorphisms iff $f$ preserves $\T$-operations of every input arity $J$ and every output arity $K$.  The following shows that we can fix $K = I$ and still obtain an equivalent condition:

\begin{PropSub}\label{thm:valued_in_thoms_quantifying_over_just_obj}
Let $A,B:\T \rightarrow \C$ be $\T$-algebras, and let $f:V \rightarrow \C(\ca{A},\ca{B})$.  Then $f$ is valued in $\T$-homomorphisms if and only if the following equations hold:
$$\Phi_{JI} \cdot f = \Psi_{JI} \cdot f\;\;\;\;\;\;\;\;(J \in \ob\J).$$
\end{PropSub}
\begin{proof}
Let us fix an object $J$ of $\J$.  It suffices to assume that the diagram \eqref{eq:diag_fam_thoms} commutes for $K = I$ and then show that the same diagram commutes for every $K \in \ob\J$.  For each $K \in \ob\J$, let us denote the lower and upper composites in \eqref{eq:diag_fam_thoms} by
$$p_K,\;q_K\;:\;\T(J,K) \longrightarrow \uV(V,\C(AJ,BK))\;,$$
respectively.  By examining the diagram \eqref{eq:diag_fam_thoms} and recalling from \ref{def:fam_t_homs} that the morphisms $f_K:V \rightarrow \C((A \circ \tau)K,(B \circ \tau)K))$ are extraordinarily $\V$-natural in $K \in \J^\op$, we deduce that the morphisms $p_K$ and $q_K$ are $\V$-natural in $K \in \J^\op$---i.e., they constitute $\V$-natural transformations
$$p,\;q\;:\;\T(J,\tau-) \Longrightarrow \uV(V,\C(AJ,(B \circ \tau)-))\;:\;\J^\op \longrightarrow \uV\;.$$
By assumption $p_I = q_I$, whereas it suffices to show that $p = q$.  But the $\V$-functors $\T(J,\tau-),\;\uV(V,\C(AJ,(B \circ \tau)-)):\J^\op \rightarrow \uV$ are $\J^\op$-algebras in $\uV$, so $p$ and $q$ are $\J^\op$-homomorphisms.   By \ref{thm:obj_of_jop_homs}, we have a fully faithful $\V$-functor $\Ev_I:\Alg{\J^\op}_{\uV} \rightarrow \uV$ that sends both $p$ and $q$ to $p_I = q_I$, so $p = q$.
\end{proof}

Using the preceding Proposition, we obtain the following strengthened variant of Theorem \bref{thm:existence_of_vcat_of_talgs}:

\begin{ThmSub}\label{thm:existence_of_vcat_talgs}\emptybox
\begin{enumerate}
\item If $\V$ has equalizers and intersections of $(\ob\J)$-indexed families of strong subobjects, then the $\V$-category of $\T$-algebras $\Alg{\T}_\C$ exists for every $\V$-category $\C$.
\item Given $\T$-algebras $A,B:\T \rightarrow \C$, the object of $\T$-homomorphisms from $A$ to $B$ is equivalently defined as a pairwise equalizer $\Alg{\T}_\C(A,B) \hookrightarrow \C(\ca{A},\ca{B})$ of the $(\ob\J)$-indexed family of parallel pairs $\Phi_{JI},\Psi_{JI}$ $(J \in \ob\J)$.
\end{enumerate}
\end{ThmSub}

\section{Commutation and Kronecker products of operations}\label{sec:cmtn}

Let $\T$ be a $\J$-theory for a given system of arities $\J \hookrightarrow \uV$.  For each pair of objects $J,K \in \ob\J = \ob\T$ we have $\V$-functors
\begin{equation}\label{eq:lr_cot_functs}[J,-]_\ell,\;\;[K,-]_r\;\;:\;\;\T \rightarrow \T\end{equation}
that supply the \textit{left cotensors} by $J$ and the \textit{right cotensors} by $K$, respectively \pbref{par:lrdes_cots}.  On objects
$$[J,K]_\ell = J \otimes K = [K,J]_r\;,$$
so we have reason to ask whether the $\V$-functors \eqref{eq:lr_cot_functs} might be the \textit{partial $\V$-functors} \cite[\S 1.4]{Ke:Ba} of a $\V$-functor in two variables
$$\T \otimes \T \rightarrow \T$$
given on objects by $(J,K) \mapsto J \otimes K$.  By \cite[(1.21)]{Ke:Ba}, this is the case if and only if the following composite morphisms are equal
\begin{equation}\label{eq:fst_kp}\T(J,J')\T(K,K') \xrightarrow{[K,-]_r[J',-]_\ell}\T(JK,J'K)\T(J'K,J'K') \xrightarrow{c} \T(JK,J'K')\end{equation}
\begin{equation}\label{eq:snd_kp}\T(J,J')\T(K,K') \xrightarrow{[K',-]_r[J,-]_\ell} \T(JK',J'K')\T(JK,JK') \xrightarrow{c} \T(JK,J'K'),\end{equation}
for all $J,J',K,K'$, where we have written the monoidal product $\otimes$ in $\V$ as juxtaposition and written $c$ to denote the relevant composition morphisms.  This leads us to the following:

\begin{DefSub}\label{def:commutation}\emptybox
\begin{enumerate}
\item For all $J,J',K,K' \in \ob\T = \ob\J$, we define the \textbf{first and second Kronecker products}
$$\kk{JJ'KK'},\;\kt{JJ'KK'}\;\;:\;\;\T(J,J')\otimes\T(K,K') \rightarrow \T(J\otimes K,J'\otimes K')$$
as the composite morphisms \eqref{eq:fst_kp} and \eqref{eq:snd_kp}, respectively.
\item Given morphisms $\mu:V \rightarrow \T(J,J')$ and $\nu:W \rightarrow \T(K,K')$ in $\V$ for objects $J,J',K,K'$ of $\J$, we call the composites
$$\mu * \nu = \left(V \otimes W \xrightarrow{\mu\otimes\nu} \T(J,J')\otimes\T(K,K') \xrightarrow{\kk{JJ'KK'}} \T(J\otimes K,J'\otimes K')\right)$$
$$\mu \stt \nu = \left(V \otimes W \xrightarrow{\mu\otimes\nu} \T(J,J')\otimes\T(K,K') \xrightarrow{\kt{JJ'KK'}} \T(J\otimes K,J'\otimes K')\right)$$
the \textbf{first and second Kronecker products} of $\mu$ and $\nu$, respectively.  When $V = W = I$, so that $\mu$ and $\nu$ are morphisms in the underlying ordinary  category $\T_0$ of $\T$, the first and second Kronecker products of $\mu$ and $\nu$ correspond to evident morphisms $J \otimes K \rightarrow J' \otimes K'$ in $\T_0$, for which we use the same notations $\mu * \nu$ and $\mu \stt \nu$.
\item We write $\mu \perp \nu$ and say that \textbf{$\mu$ commutes with $\nu$} (in $\T$) if
$$\mu * \nu = \mu \stt \nu\;:\;V\otimes W \rightarrow \T(J\otimes K,J'\otimes K'),$$
i.e., if the first and second Kronecker products of $\mu$ and $\nu$ are equal.
\end{enumerate}
\end{DefSub}

\begin{RemSub}\label{rem:diagr_text_compn}
For any triple of objects $A,B,C$ in $\V$-category $\C$ we have composition morphisms $\C(A,B) \otimes \C(B,C) \rightarrow \C(A,C)$ and $\C(B,C) \otimes \C(A,B) \rightarrow \C(A,C)$ that are related to one another by composition with the symmetry in $\V$.  We shall call these the \textbf{diagrammatic} and \textbf{textual} composition morphisms, respectively.  Observe that the first Kronecker product in a $\J$-theory $\T$ involves diagrammatic composition, whereas the second Kronecker product involves textual composition.  The repercussions of this will be evident in Example \bref{exa:mro} and implicit in Example \bref{exa:kp_matr}.
\end{RemSub}

\begin{ExaSub}[\textbf{Kronecker products of operations in Lawvere theories}]
For the system of arities $\FinCard \hookrightarrow \Set$, the Kronecker products defined in \bref{def:commutation} can be characterized in terms of the Kronecker products\footnote{At present, the term \textit{Kronecker product} in this sense does not seem to be in widespread use in the literature on algebraic theories, despite its use in the classical case of matrices \pbref{exa:kp_matr}.  Nevertheless the closely related \textit{tensor product} of theories \cite[\S 13]{Wra:AlgTh} is often called the \textit{Kronecker product of theories}, as distinguished from the above \textit{Kronecker products of operations}.} $\mu * \nu,\mu \stt \nu:j \times k \rightarrow j' \times k'$ of pairs of individual morphisms $\mu:j \rightarrow j'$, $\nu:k \rightarrow k'$ in the Lawvere theory $\T$, for which explicit formulas are given in \cite[\S 4]{Lu:CvxAffCmt}.  Here, the objects $j,j',k,k'$ are finite cardinals, and the product $j \times k$ is the usual product of cardinals $jk$, with chosen product projections in $\FinCard$ \cite[4.1]{Lu:CvxAffCmt}.
\end{ExaSub}

\begin{ExaSub}[\textbf{The Kronecker product of matrices}]\label{exa:kp_matr}
Given a rig $R$, recall that the Lawvere theory of left $R$-modules is the category $\T = \Mat_R$ of $R$-matrices, whose morphisms $j \rightarrow j'$ are $j' \times j$-matrices.  Letting $X \in \Mat_R(j,j') = R^{j'\times j}$ and $Y \in \Mat_R(k,k') = R^{k' \times k}$, the first Kronecker product $X * Y$ is the classical \textbf{Kronecker product} $Y \otimes X$ of the matrices $Y$ and $X$ \cite[4.4]{Lu:CvxAffCmt}, which is a certain $j'k' \times jk$-matrix whose entries are products of entries drawn from $Y$ and $X$.  The second Kronecker product $X \stt Y$ is in general distinct, but coincides with $Y \otimes X$ when $R$ is commutative \cite[4.6]{Lu:CvxAffCmt}.
\end{ExaSub}

\begin{ExaSub}[\textbf{Multiplication in a ring and its opposite}]\label{exa:mro}
Given a monoid $R$ in $\V$ (e.g. a ring if $\V = \Ab$), our convention is to consider $R$ as a one-object $\V$-category $\R$ whose unique \textit{textual} composition morphism \pbref{rem:diagr_text_compn} is the multiplication morphism $m_R:R \otimes R \rightarrow R$ carried by $R$.  $\R$ is then an $\{I\}$-theory \pbref{exa:sys_ar_jth}, and its first Kronecker product $\mathsf{k}$ has exactly one component, namely the \textit{diagrammatic} composition morphism carried by $\R$, i.e. the multiplication morphism $m_{R^\op}:R \otimes R \rightarrow R$ carried by the opposite monoid $R^\op$.  Contrastingly, the unique component of the second Kronecker product $\widetilde{\mathsf{k}}$ for $\R$ is the multiplication morphism $m_R$ carried by $R$ itself.
\end{ExaSub}

We shall employ the following lemma in order to establish a basic relation between the first and second Kronecker products.

\begin{LemSub}
Given objects $J,K,K'$ of $\J$, we have a commutative diagram
$$
\xymatrix{
\T(K,K') \ar[dr]_{[J,-]_r} \ar[r]^(.4){[J,-]_\ell} & \T(J\otimes K,J\otimes K') \ar[d]^{\T(\tau(s_{JK}),\tau(s_{K'J}))}\\
& \T(K\otimes J,K'\otimes J)
}
$$
in which the right side is an isomorphism.
\end{LemSub}
\begin{proof}
We have cotensors $[J,K]_\ell = J \otimes K$ and $[J,K]_r = K \otimes J$ of $K$ by $J$ in $\T$, and by \bref{par:lrdes_cots} the induced isomorphism $[J,K]_\ell \xrightarrow{\sim} [J,K]_r$ is $\tau(s_{KJ}):J \otimes K \rightarrow K \otimes J$.   Similar remarks apply with $K'$ in place of $K$, so the result follows by \bref{thm:cots_of_same_obj} since $s_{KJ}^{-1} = s_{JK}$.
\end{proof}

The first and second Kronecker products are related in the following way:

\begin{PropSub}\label{thm:reln_betw_fst_and_snd_kr_prods}
Given objects $J,J',K,K'$ of $\J$, we have a commutative square
$$
\xymatrix{
\T(J,J')\otimes\T(K,K') \ar[d]_s \ar[rr]^{\kt{JJ'KK'}} & & \T(J\otimes K, J'\otimes K') \ar[d]^{\T(\tau(s_{JK}),\tau(s_{K'J'}))}\\
\T(K,K')\otimes\T(J,J') \ar[rr]_{\kk{KK'JJ'}} & & \T(K\otimes J,K'\otimes J')
}
$$
whose left and right sides are isomorphisms.  Here $s$ denotes the symmetry isomorphism in $\V$.
\end{PropSub}
\begin{proof}
Apply the definitions of $\kk{}$ and $\kt{}$, together with the preceding Lemma.
\end{proof}

\begin{PropSub}\label{thm:cmtn_symm}
The commutation relation $\bot$ is symmetric.  I.e.,
$$\mu \bot \nu \;\;\;\;\Longleftrightarrow\;\;\;\; \nu \bot \mu\;.$$
\end{PropSub}
\begin{proof}
With $\mu$ and $\nu$ as in \bref{def:commutation}, suppose that $\mu \bot \nu$.  Then $\mu * \nu = \mu \stt \nu:V\otimes W \rightarrow \T(J\otimes K,J'\otimes K')$.  Two separate applications of \bref{thm:reln_betw_fst_and_snd_kr_prods} show not only that
$$\nu \stt \mu \;\cong\; \mu * \nu \;=\; \mu \stt \nu \;\cong\; \nu * \mu$$
in the arrow category of $\V$, but moreover that in fact the composite isomorphism is an identity $\nu \stt \mu = \nu * \mu$.
\end{proof}

\begin{DefSub}
A $\J$-theory $\T$ is \textbf{commutative} if its first and second Kronecker products are equal, i.e., if $\kk{JJ'KK'} = \kt{JJ'KK'}$ for all objects $J,J',K,K'$ of $\J$.  Equivalently, $\T$ is commutative iff $\mu$ commutes with $\nu$ for all objects $J,J',K,K'$ and all morphisms $\mu:V \rightarrow \T(J,J')$ and $\nu:W \rightarrow \T(K,K')$ in $\V$.  Indeed, note that $\kk{JJ'KK'} = 1_{\T(J,J')} * 1_{\T(K,K')}$ and $\kt{JJ'KK'} = 1_{\T(J,J')} \stt 1_{\T(K,K')}$.
\end{DefSub}

\begin{ExaSub}[\textbf{$R$-modules and commutativity}]\label{exa:second_kp_matr}
The Lawvere theory $\Mat_R$ of left $R$-modules for a rig $R$ is commutative if and only if $R$ is commutative \cite[4.6]{Lu:CvxAffCmt}.  In particular, the Lawvere theory $\Mat_2$ of semilattices \pbref{exa:slats} is commutative.
\end{ExaSub}

\begin{ExaSub}[\textbf{Commutative rings as commutative} $\{\ZZ\}$-\textbf{theories}]\label{exa:crings_comm_zth}
By \bref{exa:mro}, commutative monoids $R$ in $\V$ are the same as commutative $\{I\}$-theories.  In particular, commutative rings are the same as commutative $\{\ZZ\}$-theories when $\V = \Ab$.
\end{ExaSub}

\begin{DefSub}[\textbf{Commutation of morphisms of theories}]
A pair of morphisms of $\J$-theories $A:\T \rightarrow \U$ and $B:\sS \rightarrow \U$ is said to \textbf{commute} if the associated morphisms
$$A_{JJ'}:\T(J,J') \rightarrow \U(J,J')\;\;\;\;\;\;B_{KK'}:\sS(K,K') \rightarrow \U(K,K')$$
commute in $\U$ for all objects $J,J',K,K'$ of $\J$.
\end{DefSub}

\begin{RemSub}
Observe that a $\J$-theory $\T$ is commutative iff the identity morphism $1_\T$ commutes with itself.
\end{RemSub}

\begin{ExaSub}[\textbf{Commutation of ring homomorphisms}]\label{exa:cmtn_ring_hom}
Let $a:R \rightarrow U$ and $b:S \rightarrow U$ be morphisms of monoids in $\V$, with corresponding morphisms of $\{I\}$-theories $A:\R \rightarrow \U$ and $B:\sS \rightarrow \U$.  Then $A$ commutes with $B$ if and only if $m_U \cdot (a \otimes b) = m_{U^\op} \cdot (a \otimes b):R \otimes S \rightarrow U$ in the notation of \bref{exa:mro}.  In particular, when $\V = \Ab$, the homomorphisms of rings $a$ and $b$ commute in this sense if and only if $a(r)b(s) = b(s)a(r)$ in $U$ for all $r \in R$ and $s \in S$.
\end{ExaSub}

\begin{PropSub}\label{thm:basic_props_of_cmmtn_morphs}
Let $P:\sP \rightarrow \T$, $Q:\Q \rightarrow \T$, and $A:\T \rightarrow \U$ be morphisms of $\J$-theories.  Firstly, if $P$ commutes with $Q$, then $AP$ commutes with $AQ$.  Secondly, if $A$ is a subtheory embedding and $AP$ commutes with $AQ$, then $P$ commutes with $Q$.
\end{PropSub}
\begin{proof}
For all $J,J',K,K' \in \ob\J$, we have a diagram
$$
\xymatrix@C=16ex @R=5ex{
\T(J,J')\T(K,K') \ar[r]^{A_{JJ'}A_{KK'}} \ar[d]_{[K,-]_r[J',-]_\ell} & \U(J,J')\U(K,K') \ar[d]^{[K,-]_r[J',-]_\ell}\\
\T(JK,J'K)\T(J'K,J'K') \ar[d]_c \ar[r]^{A_{JK,J'K}A_{J'K,J'K'}} & \U(JK,J'K)\U(J'K,J'K') \ar[d]^c\\
\T(JK,J'K') \ar[r]_{A_{JK,J'K'}} & \U(JK,J'K')
}
$$
in which we have written the monoidal product $\otimes$ in $\V$ as juxtaposition.  The upper square commutes by \eqref{eq:mor_com_lcot} and \eqref{eq:mor_com_rcot}, and the lower square commutes by the $\V$-functoriality of $A$.  But the composites on the left and right sides are the first Kronecker products $\kk{JJ'KK'}^\T$ and $\kk{JJ'KK'}^\U$ for $\T$ and $\U$, respectively.  It suffices to show firstly that if
\begin{enumerate}
\item[(i)] $\kk{JJ'KK'}^\T \cdot (P_{JJ'} \otimes Q_{KK'}) = \kt{JJ'KK'}^\T \cdot (P_{JJ'} \otimes Q_{KK'})$
\end{enumerate}
then
\begin{enumerate}
\item[(ii)] $\kk{JJ'KK'}^\U \cdot ((AP)_{JJ'} \otimes (AQ)_{KK'}) = \kt{JJ'KK'}^\U \cdot ((AP)_{JJ'} \otimes (AQ)_{KK'})\;,$
\end{enumerate}
and secondly that (ii) implies (i) when $A$ is a subtheory embedding.  But by the above we now know that the left-hand side of (ii) is
$$\kk{JJ'KK'}^\U \cdot (A_{JJ'} \otimes A_{KK'}) \cdot (P_{JJ'} \otimes Q_{KK'}) = A_{JK,J'K'} \cdot \kk{JJ'KK'}^\T \cdot (P_{JJ'} \otimes Q_{KK'})\;,$$
and we find similarly that the right-hand side of (ii) is
$$\kt{JJ'KK'}^\U \cdot (A_{JJ'} \otimes A_{KK'}) \cdot (P_{JJ'} \otimes Q_{KK'}) = A_{JK,J'K'} \cdot \kt{JJ'KK'}^\T \cdot (P_{JJ'} \otimes Q_{KK'})\;.$$
The result now follows.
\end{proof}

\begin{PropSub}
Any subtheory $\T$ of a commutative $\J$-theory $\U$ is commutative.
\end{PropSub}
\begin{proof}
Letting $A:\T \hookrightarrow \U$ be a subtheory embedding, the commutativity of $\U$ immediately entails that $A$ commutes with itself, but since $A = A \circ 1_\T$ and $A$ is a subtheory embedding, it follows from \bref{thm:basic_props_of_cmmtn_morphs} that $1_\T$ commutes with itself.
\end{proof}

\begin{ExaSub}[\textbf{Affine and convex spaces}]\label{exa:commutative_lths}
The Lawvere theory $\Mat_R^\aff$ of $R$-affine spaces \pbref{exa:raff} for a commutative ring or rig $R$ is commutative, as it is a subtheory of the commutative theory $\Mat_R$ of $R$-modules \pbref{exa:second_kp_matr}.  In particular, the theory of $\RR$-convex spaces $\Mat^\aff_{\RR_+}$ \pbref{exa:cvx_sp} is commutative, as is the theory of unbounded join semilattices $\Mat_2^\aff$ \pbref{exa:unb_slat}.
\end{ExaSub}

\section{Commutation via \texorpdfstring{$\T$}{T}-homomorphisms}

In the present section we establish a link between commutation and the notion of $\T$-homomorphism.  The connection between these notions will play a fundamental role in our study of \textit{commutants} in subsequent sections.  We begin with some technical lemmas, as follows.

\begin{LemSub}\label{thm:lrcot_vfuncs_for_jop}
For each object $J$ of $\J$, the $\V$-functors
$$[J,-]_\ell,\;[J,-]_r\;:\;\J^\op \rightarrow \J^\op$$
are simply $J \otimes (-)$ and $(-)\otimes J$, respectively.
\end{LemSub}
\begin{proof}
The left $\J$-cotensor counits $\gamma^K_J = \Coev:J \rightarrow \J^\op(J\otimes K,K) = \uV(K,J\otimes K)$ \pbref{par:lrdes_cots} are (extraordinarily) $\V$-natural in $K \in \J^\op$ with respect to the $\V$-functor $J\otimes (-)$.  But by \bref{par:cot}, $[J,-]_\ell$ is the unique $\V$-endofunctor on $\J^\op$ that is given on objects by $K \mapsto J\otimes K$ and makes the $\gamma^K_J$ $\V$-natural in $K \in \J^\op$, so $[J,-]_\ell = J \otimes (-)$.  By a similar argument, $[J,-]_r = (-)\otimes J$.
\end{proof}

\begin{LemSub}\label{thm:der_cot_via_left_cot}
For all $J,K,K' \in \ob\J$, the diagram
$$
\xymatrix{
J \ar[d]_{\gamma^K_J} \ar[rr]^{\gamma_J} & & \T(J,I) \ar[d]^{[K,-]_r}\\
\T(J\otimes K,K) \ar[rr]_{\T(1,\tau(\ell_K))} & & \T(J\otimes K,I\otimes K)
}
$$
commutes, where $\gamma_J$ is the counit for the designated cotensor $[J,I] = J$ in $\T$, $\gamma^K_J$ is the counit for the left cotensor $[J,K]_\ell = J\otimes K$ in $\T$, and $\tau(\ell_K):K \xrightarrow{\sim} I \otimes K$ is the isomorphism in $\T$ obtained from the isomorphism $\ell_K:I \otimes K \rightarrow K$ in $\J$ by applying $\tau:\J^\op \rightarrow \T$.
\end{LemSub}
\begin{proof}
By \bref{def:morph_th} and \bref{rem:morph_pres_lr_jcots}, $\tau$ strictly preserves the designated cotensors $[J,I] = J$ as well as all the left and right $\J$-cotensors, so we readily reduce to the case of $\T = \J^\op$.  In this case, \bref{thm:lrcot_vfuncs_for_jop} entails that the diagram in question is simply
$$
\xymatrix{
J \ar[d]_{\Coev} \ar[r]^{\gamma_J} & \uV(I,J) \ar[d]^{(-)\otimes K}\\
\uV(K,J\otimes K) \ar[r]_(.45){\uV(\ell,1)} & \uV(I\otimes K,J\otimes K)
}
$$
recalling that $\gamma_J$ here is the transpose of $r_J:J\otimes I \rightarrow J$.  Upon taking transposes of the two composites in this diagram, we obtain the morphisms
$$J \otimes \ell_K,\;r_J\otimes K\;:\;J\otimes I \otimes K \rightarrow J \otimes K$$
which are equal, by one of the axioms for monoidal categories (MC2 of \cite[II.1]{EiKe}).
\end{proof}

\begin{LemSub}\label{thm:lemma_on_cmtn_vs_thoms}
Let $A:\T \rightarrow \U$ and $B:\sS \rightarrow \U$ be morphisms of $\J$-theories, and let $J,J',K,K'$ be objects of $\J$.   Write $\upsilon:\J^\op \rightarrow \U$ for the unique morphism of theories.  Then the following conditions are equivalent:
\begin{enumerate}
\item $A_{JJ'}:\T(J,J') \rightarrow \U(J,J')$ commutes with $B_{KK'}:\sS(K,K') \rightarrow \U(K,K')$.
\item The composite
$$\sS(K,K') \xrightarrow{B_{KK'}} \U(K,K') \xrightarrow{\theta_{KK'}} \U(I\otimes K,I\otimes K')$$
preserves $\T$-operations of input arity $J$ and output arity $J'$ \pbref{def:pres_ops}, where the objects $I \otimes K$ and $I \otimes K'$ of $\U$ are considered here as the carriers of the $\T$-algebras $[K,A]_r,[K',A]_r:\T \rightarrow \U$ \pbref{par:cot_alg}, respectively, and $\theta_{KK'}$ is defined as the isomorphism $\U(\upsilon(\ell^{-1}_K),\upsilon(\ell_{K'}))$.
\end{enumerate}
\end{LemSub}
\begin{proof}
Recall from \bref{par:cot_alg} that the pointwise right cotensor $[K,A]_r$ of $A$ in $\Alg{\T}_\U$ is the composite $[K,-]_r \circ A$ of $A$ with the $\V$-functor $[K,-]_r:\U \rightarrow \U$, and similarly for $[K',A]$.  Therefore it follows immediately from the definition that $A_{JJ'}$ and $B_{KK'}$ commute if and only if the diagram
\begin{equation}\label{eq:diag_commt_morph_via_rcot}
\xymatrix{
\T(J,J')\sS(K,K') \ar[d]_{1\otimes B_{KK'}} \ar[r]^{1 \otimes B_{KK'}} & \T(J,J')\U(K,K') \ar[d]^{[K,A]_r[J',-]_\ell}\\
\T(J,J')\U(K,K') \ar[d]_{[K',A]_r[J,-]_\ell} & \U(JK,J'K)\U(J'K,J'K') \ar[d]^c\\
U(JK',J'K')\U(JK,JK') \ar[r]_c & \U(JK,J'K')
}
\end{equation}
commutes, where we have omitted some subscripts and written $\otimes$ as juxtaposition.

On the other hand, condition 2 is (by definition) equivalent to the commutativity of a diagram of the form \eqref{eq:diag_fam_thoms_ii}, and one finds that this diagram is almost exactly the same as \eqref{eq:diag_commt_morph_via_rcot}, except that one must substitute the composites
\begin{equation}\label{eq:lambdaj_of_theta}\U(K,K') \xrightarrow{{\theta_{KK'}}} \U(IK,IK') \xrightarrow{\lambda^{[K,A]_r[K',A]_r}_J} \U(JK,JK')\end{equation}
$$\U(K,K') \xrightarrow{{\theta_{KK'}}} \U(IK,IK') \xrightarrow{\lambda^{[K,A]_r[K',A]_r}_{J'}} \U(J'K,J'K')$$
in place of the morphisms $[J,-]_\ell$ and $[J',-]_\ell$ that appear in \eqref{eq:diag_commt_morph_via_rcot}, noting that $[K,A]_r(I) = IK$, $[K,A]_r(J) = JK$, and similarly with $J',K'$ in place of $J,K$.  Here the morphisms $\lambda^{[K,A]_r[K',A]_r}_J,\lambda^{[K,A]_r[K',A]_r}_{J'}$ are as defined in \bref{thm:obj_of_jop_homs}.

Hence it suffices to show that the composite \eqref{eq:lambdaj_of_theta} equals $[J,-]_\ell$ for all objects $J,K,K'$ of $\J$.  This we will accomplish through a suitable invocation of \bref{thm:cot_structs_on_same_objs}.  First observe that since $[K,A]_r = [K,-]_r \circ A$ and $[K,A]_r$ is a $\T$-algebra, the composite
\begin{equation}\label{eq:counit_cot_ik_by_j}\tilde{\gamma}^K_J := \left(J \xrightarrow{\gamma_J} \T(J,I) \xrightarrow{A_{JI}} \U(J,I) \xrightarrow{([K,-]_r)_{JI}} \U(JK,IK)\right)\end{equation}
presents $JK$ as a cotensor $[J,IK]$ of $IK$ by $J$ in $\U$.  But $[J,K]_\ell = JK$ is also a cotensor of $K \cong IK$ by $J$ in $\U$, so the isomorphism $\upsilon(\ell_K):K \rightarrow IK$ induces an isomorphism $[J,\upsilon(\ell_K)]:[J,K]_\ell \rightarrow [J,IK] = JK$.

We claim that this induced isomorphism $[J,\upsilon(\ell_K)]$ is the identity arrow on $JK$.  Indeed, the counit \eqref{eq:counit_cot_ik_by_j} is equally the composite
$$\tilde{\gamma}^K_J = \left(J \xrightarrow{\gamma_J} \U(J,I) \xrightarrow{([K,-]_r)_{JI}} \U(JK,IK)\right)$$
since $A$ strictly preserves the designated cotensors $[J,I] = J$, and by \bref{thm:der_cot_via_left_cot} this composite can be re-expressed as
\begin{equation}\label{eq:char_gamma_tilde}\tilde{\gamma}^K_J = \left(J \xrightarrow{\gamma_J^K} \U(JK,K) \xrightarrow{\U(1,\upsilon(\ell_K))} \U(JK,IK)\right).\end{equation}

Similar remarks apply with $K'$ in place of $K$, and by definition the morphism $\lambda^{[K,A]_r[K',A]_r}_J$ appearing in \eqref{eq:lambdaj_of_theta} is the morphism $[J,-]:\U(IK,IK') \rightarrow \U(JK,JK')$ induced by the cotensors $[J,IK] = JK$ and $[J,IK'] = JK'$.  We can now invoke \bref{thm:cot_structs_on_same_objs} to deduce that the composite \eqref{eq:lambdaj_of_theta} equals $[J,-]_\ell$, as needed.
\end{proof}

\begin{ThmSub}\label{thm:cmtn_via_thoms}
Let $A:\T \rightarrow \U$ and $B:\sS \rightarrow \U$ be morphisms of $\J$-theories.  Then the following are equivalent:
\begin{enumerate}
\item $A$ commutes with $B$.
\item For all objects $K,K'$ of $\J$, $B_{KK'}:\sS(K,K') \rightarrow \U(K,K')$ is valued in $\T$-homomorphisms \pbref{def:fam_t_homs} between the pointwise designated cotensors $[K,A]$ and $[K',A]$ of the $\T$-algebra $A$ \pbref{par:cot_alg}.
\end{enumerate}
\end{ThmSub}
\begin{proof}
By \bref{thm:lemma_on_cmtn_vs_thoms} we know that 1 is equivalent to the statement that for all $K,K' \in \ob\J$, $\theta_{KK'} \cdot B_{KK'}$ is valued in $\T$-homomorphisms from $[K,A]_r$ to $[K',A]_r$.  But we have isomorphisms of $\T$-algebras $\alpha:[K,A] \rightarrow [K,A]_r$ and $\beta:[K',A] \rightarrow [K',A]_r$ whose underlying morphisms $\ca{\alpha}$ and $\ca{\beta}$ in $\U$ are the canonical isomorphisms $[K,I] \rightarrow [K,I]_r$ and $[K',I] \rightarrow [K',I]_r$ between the \textit{designated} and \textit{right} cotensors of $I$ by $K$ and $K'$ \pbref{par:lrdes_cots}.  It is straightforward to show that these isomorphisms in $\U$ are $\upsilon(\ell_K):K \rightarrow I \otimes K$ and $\upsilon(\ell_{K'}):K' \rightarrow I \otimes K'$, respectively, in the notation of \bref{thm:lemma_on_cmtn_vs_thoms}, so since $\theta_{KK'} = \U(\upsilon(\ell^{-1}_K),\upsilon(\ell_{K'})) = \U(\ca{\alpha}^{-1},\ca{\beta})$ the result now follows.
\end{proof}

\section{Commutants}\label{sec:cmtnt}

Let $\T$ and $\U$ denote $\J$-theories for which the $\V$-category $\Alg{\T}_\U$ of $\T$-algebras in $\U$ exists.  Recall that any morphism of $\J$-theories $A:\T \rightarrow \U$ is, in particular, a $\T$-algebra in $\U$.

\begin{DefSub}\label{def:cmtnt}
Given a morphism of $\J$-theories $A:\T \rightarrow \U$, the \textbf{commutant} $\T^\bot_A$ of $\T$ with respect to $A$, also called the \textit{commutant of $A$}, is the full $\J$-theory of $A$ in $\Alg{\T}_\U$.  In symbols,
$$\T^\bot_A = (\Alg{\T}_\U)_A\;.$$
Explicitly,
\begin{equation}\label{eq:cmtnt_hom}\T^\bot_A(J,K) = \Alg{\T}_\U([J,A],[K,A])\;\;\;\;\;\;(J,K \in \ob\J),\end{equation}
where $[J,A],[K,A]:\T \rightarrow \U$ are the cotensors in $\Alg{\T}_\U$ \pbref{par:cot_alg}.  Even if $\Alg{\T}_\U$ does not exist, we can clearly still define the commutant $\T^\bot_A$ as soon as the relevant objects of $\T$-homomorphisms \eqref{eq:cmtnt_hom} exist, in which case we say that \textbf{the commutant exists}.
\end{DefSub}

\begin{ThmSub}\label{thm:existence_of_commutant_via_intersections}
If $\V$ has equalizers and intersections of $(\ob\J)$-indexed families of strong subobjects, then the commutant of any morphism of $\J$-theories exists.
\end{ThmSub}
\begin{proof}
This follows immediately from \bref{thm:existence_of_vcat_talgs}.
\end{proof}

\begin{DefSub}
A $\J$-\textbf{theory over} $\U$ is a $\J$-theory $\T$ equipped with a morphism $\T \rightarrow \U$.  Given a $\J$-theory $\T$ over $\U$, we denote the commutant of the associated morphism $\T \rightarrow \U$ as simply $\T^\perp$ and call it the \textbf{commutant of} $\T$.  Similarly, given $\J$-theories $\T$ and $\sS$ over $\U$, we say that $\T$ and $\sS$ \textbf{commute} if their associated morphisms to $\U$ commute, in which case we write $\T \perp \sS$.
\end{DefSub}

\begin{RemSub}
It is helpful to consider the case of a subtheory $\T \hookrightarrow \U$, in which case we also call $\T^\perp$ the \textit{commutant of $\T$ in $\U$}.  Fittingly, $\T^\perp$ is always a subtheory of $\U$, even when $\T$ is not:
\end{RemSub}

\begin{PropSub}\label{thm:cmtnt_is_subth}
Given a morphism of $\J$-theories $A:\T \rightarrow \U$, the commutant $\T^\bot_A$ is a strong subtheory of $\U$.
\end{PropSub}
\begin{proof}
Let $\iota$ denote the composite $\V$-functor
$$\T^\bot_A = (\Alg{\T}_\U)_A \overset{i}{\rightarrowtail} \Alg{\T}_\U \xrightarrow{|-|} \U$$
whose first factor $i$ is the canonical identity-on-homs $\V$-functor \pbref{def:full_theory} and whose second factor $\ca{\text{$-$}}$ is the `forgetful' $\V$-functor \pbref{par:vcat_talgs}.  Taking the $\J$-cotensors in \eqref{eq:cmtnt_hom} to be the pointwise designated cotensors \pbref{par:cot_alg}, it follows that $\ca{\text{$-$}}$ strictly preserves the designated $\J$-cotensors.  But $i$ is a normal $\T^\perp_A$-algebra \pbref{def:full_theory}, so the composite $\iota$ is a normal $\T^\perp_A$-algebra with carrier $\iota(I) = \ca{A} = I$, equivalently, a morphism of $\J$-theories \pbref{def:morph_th}.  Further, $\iota$ is strongly faithful since $\ca{\text{$-$}}$ is strongly faithful \pbref{rem:carrier_functor_str_ff}.
\end{proof}

\begin{ExaSub}[\textbf{The commutant or centralizer of a subring}]\label{exa:cmtnt_subring}
Let $a:R \rightarrow U$ be a morphism of monoids in $\V$, with corresponding morphism of $\{I\}$-theories $A:\R \rightarrow \U$.  Then the commutant of $A$ is a submonoid $R^\perp_a \hookrightarrow U$, namely the equalizer of the pair of morphisms $\Psi_{II},\Phi_{II}:U \rightarrow \uV(R,U)$ (in the notation of \ref{eq:pairs_specifying_obj_thoms}) obtained as transposes of the composites $m_{U} \cdot (a \otimes 1_U),\;m_{U^\op} \cdot (a \otimes 1_U)\;:\;R \otimes U \rightarrow U$ in the notation of \bref{exa:mro}, \bref{exa:cmtn_ring_hom}.  When $\V = \Ab$, so that $a$ is a homomorphism of rings, $R^\perp_a \subseteq U$ is the familiar centralizer (or commutant) of the image $a(R) \subseteq U$ of $a$.
\end{ExaSub}

\begin{ExaSub}[\textbf{Commutants for Lawvere theories}]
When $\V = \Set$ and $\J = \FinCard$ we recover the notion of commutant for Lawvere theories that is studied in \cite{Lu:CvxAffCmt} and is due to Wraith \cite{Wra:AlgTh}, who defined a similar notion of commutant for Linton's equational theories \cite{Lin:Eq} (i.e. $\J$-theories with $\J = \V = \Set$).  By \cite[5.6, 5.9]{Lu:CvxAffCmt}, the commutant of a subtheory $\T$ of a Lawvere theory $\U$ is the subtheory $\T^\perp \hookrightarrow \U$ consisting of those morphisms $\mu$ of $\U$ with the property that $\mu$ commutes with every morphism $\nu$ of $\T$.
\end{ExaSub}

The link that was established in \bref{thm:cmtn_via_thoms} between commutation and the notion of $\T$-homomorphism now enables us to make the connection between commutants and commutation in our general context:

\begin{ThmSub}\label{thm:commutants_via_commutativity}
Let $A:\T \rightarrow \U$ and $B:\sS \rightarrow \U$ be morphisms of $\J$-theories.  Then $A$ and $B$ commute if and only if $B$ factors through the commutant $\T^\perp_A \hookrightarrow \U$ of $A$.
\end{ThmSub}
\begin{proof}
$B$ factors through $\T^\perp_A \hookrightarrow \U$ if and only if each of its components $B_{KK'}$ factors through the subobject
\begin{equation}\label{eq:commutant_incl}\T^\perp_A(K,K') = \Alg{\T}_\U([K,A],[K',A]) \hookrightarrow \U(K,K')\;,\end{equation}
where $[K,A]$ and $[K',A]$ are the pointwise designated cotensors.  This holds if and only if each component $B_{KK'}$ is valued in $\T$-homomorphisms from $[K,A]$ to $[K',A]$, so the result follows from \bref{thm:cmtn_via_thoms}.
\end{proof}

\begin{CorSub}
Given a $\J$-theory $\T$ over $\U$, the commutant $\T^\perp \hookrightarrow \U$ is the largest subtheory of $\U$ that commutes with $\T$.
\end{CorSub}
\begin{proof}
By the preceding theorem, a subtheory $\sS \hookrightarrow \U$ commutes with $\T$ if and only if $\sS$ is contained in $\T^\perp$, i.e., iff $\sS \hookrightarrow \U$ factors through $\T^\perp \hookrightarrow \U$.  In particular, $\T^\perp \hookrightarrow \U$ therefore commutes with $\T$.
\end{proof}

\begin{DefSub}\label{def:cmtnt_talg}
Given a $\T$-algebra $A:\T \rightarrow \C$, the \textbf{commutant} $\T^\bot_A$ of $A$ is defined as the commutant of the associated morphism of $\J$-theories $A:\T \rightarrow \C_A$ (where $\C_A$ is the full $\J$-theory of $A$ in $\C$, \bref{def:full_theory}).  Equivalently, $\T^\perp_A$ is the full $\J$-theory of $A$ in $\Alg{\T}_\C$, provided that the latter $\V$-category exists.  By \bref{thm:cmtnt_is_subth}, $\T^\perp_A$ is a strong subtheory of $\C_A$.  By \bref{def:full_theory} we have a fully faithful $\C_A$-algebra $\C_A \rightarrowtail \C$ with carrier $\ca{A}$, and so the composite $\T^\perp_A \hookrightarrow \C_A \rightarrow \C$ is a $\T^\perp_A$-algebra that we call the \textbf{canonical} $\T^\perp_A$-\textbf{algebra}.  Observe that the canonical $\T^\perp_A$-algebra has the same carrier as $A$ itself.
\end{DefSub}

\begin{ExaSub}[\textbf{The commutant of an $R$-module when $\V = \Ab$}]\label{exa:cmtnt_rmod_over_ab}
Let $R$ be a ring and $M$ a left $R$-module $M$.  We can view $M$ equally as an $\R$-algebra for the $\{\ZZ\}$-theory $\R$ corresponding to $R$, and then the  commutant $R^\perp_M := \R^\perp_M$ of $M$ is the commutant of the morphism of rings $R \rightarrow \End_\ZZ(M)$ determined by $M$, where $\End_\ZZ(M)$ denotes the ring of endomorphisms of the abelian group underlying $M$.  Hence $R^\perp_M$ is the subring $\End_R(M)$ of $\End_\ZZ(M)$ consisting of all left $R$-linear maps.
\end{ExaSub}

\begin{ExaSub}[\textbf{The Lawvere theory of left $R$-modules}]\label{exa:cmtnt_lth_rmods}
Let $R$ be a ring or rig, and let $\T = \Mat_R$ be the Lawvere theory of left $R$-modules \pbref{exa:law_th_rmods}.  $R$ itself is a left $R$-module, equivalently, a normal $\T$-algebra, and the corresponding morphism of theories $R:\T \rightarrow \Set_R$ \pbref{def:full_theory} presents $\T$ as a theory over the full finitary theory $\Set_R$ of $R$ in $\Set$.  It is proved in \cite[5.14]{Lu:CvxAffCmt} that the commutant $\T^\perp$ of $\T = \Mat_R$ over $\Set_R$ is (isomorphic to) the theory $\Mat_{R^\op}$ of \textit{right} $R$-modules.
\end{ExaSub}

\begin{RemSub}\label{rem:jalg_str}
Generalizing Lawvere's notion of the \textit{algebraic structure} of a set-valued functor $U:\B \rightarrow \Set$ \cite[III.1]{Law:PhD}, we can define the $\J$-\textbf{algebraic structure} $\textnormal{Str}(U)$ of a $\V$-functor $U:\B \rightarrow \uV$ as the full $\J$-theory of $U$ in the $\V$-functor $\V$-category $[\B,\uV]$, if the latter exists; more generally we, can still similarly define $\textnormal{Str}(U)$ as soon as the objects of $\V$-natural transformations 
$$\textnormal{Str}(U)(J,K) = [\B,\uV]([J,U],[K,U])\;\;\;\;\;\;\;(J,K \in \ob\J)$$
exist, where $[J,U]$ denotes the pointwise cotensor.  The case where $\J = \uV$ was studied by Dubuc \cite{Dub:EnrStrSem}.  Lawvere showed that the \textit{structure} functor $\textnormal{Str}$ is left adjoint to \textit{semantics}---the passage from a theory to its category of algebras, equipped with its canonical functor to $\Set$---and Dubuc established an analogous result in the $\J = \uV$ case.

Note that the notion of commutant intersects with the above notion of $\J$-algebraic structure:  Indeed, the commutant of a $\uV$-valued $\T$-algebra $A:\T \rightarrow \uV$ is equally the $\J$-algebraic structure $\textnormal{Str}(A)$ of $A$.  On the other hand, the notion of commutant applies to $\T$-algebras $A:\T \rightarrow \C$ valued in an arbitrary $\V$-category $\C$, rather than just $\C = \uV$.  Clearly one can immediately generalize the above notion of $\J$-algebraic structure to apply to any such $\C$, but the relation of structure and semantics has not been studied in this context within the literature\footnote{But see Linton's related work \cite{Lin:OutlFuncSem} in the non-enriched context.}.  Furthermore, the theory of commutants has a different character in several respects, as is particularly evident in \S \bref{sec:self_adj_cmtnt_func}.  It is also notable that one has strong general existence results for the commutant of a morphism of $\J$-theories as soon as certain wide intersections and equalizers exist in $\V$ \pbref{thm:existence_of_commutant_via_intersections}, and in the case of a $\uV$-valued $\T$-algebra $A$ we shall establish below a further result to effect that the commutant $\T^\perp_A = \textnormal{Str}(A)$ always exists for many systems of arities $\J$ \pbref{thm:cmtnt_alg_ele} including $\J = \uV$ when $\V$ has equalizers.
\end{RemSub}

Suppose $\T$ is a $\J$-theory for which the $\V$-category of $\T$-algebras in $\T$ exists.

\begin{DefSub}
The \textbf{centre} of the $\J$-theory $\T$ is the commutant of $\T$ in itself, i.e. the commutant $Z(\T) := \T^\perp_{1_\T}$ of the identity morphism on $\T$.  A morphism of $\J$-theories $A:\sS \rightarrow \T$ is \textbf{central} if it commutes with the identity morphism on $\T$.  Hence $A$ is central iff $A$ factors through the centre $Z(\T) \hookrightarrow \T$.  Note that $\T$ is commutative if and only if it is isomorphic to its centre (as a subtheory of $\T$).
\end{DefSub}

\begin{PropSub}\label{thm:unique_morphism_from_jop_is_central}
The unique morphism $\tau:\J^\op \rightarrow \T$ is central.  Therefore, the commutant of $\tau$ is isomorphic to $\T$.
\end{PropSub}
\begin{proof}
There is a unique morphism of $\J$-theories $z:\J^\op \rightarrow Z(\T)$, and since the subtheory embedding $\iota:Z(\T) \hookrightarrow \T$ is a morphism of $\J$-theories, we have $\iota \circ z = \tau$.
\end{proof}

\section{The self-adjoint commutant functor}\label{sec:self_adj_cmtnt_func}

Let $\U$ be a $\J$-theory for which the commutant of each $\J$-theory over $\U$ exists.  For example, this is true for every $\J$-theory $\U$ as soon as $\V$ has equalizers and intersections of $(\ob\J)$-indexed families of strong subobjects \pbref{thm:existence_of_commutant_via_intersections}.

\begin{DefSub}
Let $\ThJ$ denote the category of all $\J$-theories and their morphisms.  We shall denote by $\ThJ \slash \U$ the \textbf{category of} $\J$-\textbf{theories over} $\U$, i.e. the slice category over $\U$ in $\ThJ$.  We denote by $\SubThJ(\U)$ the full subcategory of $\ThJ\slash\U$ consisting of all subtheories of $\U$.
\end{DefSub}

\begin{RemSub}\label{rem:uniq_morph_into_subth}
Observe that for theories $\T$ and $\sS$ over $\U$, if $\sS$ is a subtheory of $\U$ then there is at most one morphism $\T \rightarrow \sS$ in the category over $\J$-theories over $\U$.  In particular, $\SubThJ(\U)$ is therefore a preordered set.  Further, we obtain the following corollary to \bref{thm:commutants_via_commutativity}: 
\end{RemSub}

\begin{PropSub}\label{thm:char_cmt_for_th_over_u}
Let $\sS$ and $\T$ be $\J$-theories over $\U$.  Then $\sS$ and $\T$ commute if and only if there is a (necessarily unique) morphism $\sS \rightarrow \T^{\perp}$ in $\ThJ\slash\U$. 
\end{PropSub}

\begin{CorSub}
For each $\J$-theory $\T$ over $\U$, there is a unique morphism
$$\eta_\T:\T \rightarrow \T^{\perp\perp}$$
in $\ThJ\slash\U$.
\end{CorSub}
\begin{proof}
Since $\T$ and $\T^\perp$ commute, this follows from the preceding Proposition.
\end{proof}

\begin{CorSub}
There is a unique functor $(-)^\perp:(\ThJ\slash\U)^\op \rightarrow \ThJ\slash\U$ that sends each $\J$-theory $\T$ over $\U$ to its commutant $\T^\perp$.
\end{CorSub}
\begin{proof}
Given a morphism $M:\sS \rightarrow \T$ in $\ThJ\slash\U$, we obtain a composite morphism
$$\sS \xrightarrow{M} \T \xrightarrow{\eta_{\T}} \T^{\perp\perp}$$
in $\ThJ\slash\U$, so by \bref{thm:char_cmt_for_th_over_u} we deduce that $\sS$ commutes with $\T^\perp$, so $\T^\perp$ commutes with $\sS$ and hence, by \bref{thm:char_cmt_for_th_over_u} again, there is a unique morphism
$$M^\perp:\T^\perp \rightarrow \sS^\perp$$
in $\ThJ\slash\U$.  In other words, $\T^\perp \lt \sS^\perp$ in the preorder $\SubThJ(\U)$, and the result follows.
\end{proof}

\begin{ThmSub}\label{thm:adjn}
There is an adjunction
$$\Adjn{\ThJ\slash\U}{(-)^\perp}{(-)^\perp}{}{}{(\ThJ\slash\U)^\op}$$
between the category of $\J$-theories over $\U$ and its opposite, in which both the left and right adjoints are given by the same contravariant functor $(-)^\perp$, which sends a $\J$-theory $\T$ over $\U$ to its commutant $\T^\perp$.
\end{ThmSub}
\begin{proof}
It suffices to show that $(\T^\perp,\eta_\T:\T \rightarrow \T^{\perp\perp})$ is a universal arrow for the putative right adjoint $(-)^\perp$.  Indeed, given a morphism $M:\T \rightarrow \sS^\perp$ in $\ThJ\slash\U$, we know by \bref{thm:char_cmt_for_th_over_u} that $\T \perp \sS$, so $\sS \perp \T$ and hence there is a unique morphism $\widetilde{M}:\sS \rightarrow \T^\perp$ in $\ThJ\slash\U$.  Further, $\widetilde{M}^\perp \cdot \eta_\T$ and $M$ are both morphisms $\T \rightarrow \sS^\perp$ in $\ThJ\slash\U$ and so, by \bref{rem:uniq_morph_into_subth}, are equal.
\end{proof}

Recall that the term \textbf{Galois connection} is an alias for the notion of adjunction for preordered sets, especially when one of the two preorders involved is presented as a dual.

\begin{CorSub}\label{thm:galois_connection}
Suppose that the system of arities $\J \hookrightarrow \uV$ admits $\V$-categories of algebras, and let $\U$ be a $\J$-theory.  Then there is a Galois connection
$$\Adjn{\SubThJ(\U)}{(-)^\perp}{(-)^\perp}{}{}{\SubThJ(\U)^\op}$$
on the preordered set $\SubThJ(\U)$ of subtheories of $\U$, given by taking the commutant $\T^\perp$ of each subtheory $\T$ of $\U$.
\end{CorSub}

\begin{DefSub}
Let $\T$ be a $\J$-theory over $\U$.
\begin{enumerate}
\item $\T$ is said to be \textbf{saturated} if $\T^{\perp\perp} \cong \T$ as theories over $\U$.
\item $\T$ is said to be \textbf{balanced} if $\T^\perp \cong \T$ as theories over $\U$.
\end{enumerate}
\end{DefSub}

\begin{RemSub}
The following are immediate consequences of the definitions:
\begin{enumerate}
\item A saturated $\J$-theory $\T$ over $\U$ is necessarily a subtheory of $\U$.
\item Any balanced $\J$-theory $\T$ over $\U$ is necessarily saturated.
\end{enumerate}
Hence we refer to saturated (resp. balanced) $\J$-theories over $\U$ equally as \textbf{saturated subtheories} (resp. \textbf{balanced subtheories}) of $\U$.
\end{RemSub}

\begin{RemSub}
We say that a subtheory $\T$ of $\U$ is commutative if $\T$ is commutative as a $\J$-theory.  Observe that by \bref{thm:basic_props_of_cmmtn_morphs}, a subtheory $\T$ of $\U$ is commutative if and only if the given embedding $\T \hookrightarrow \U$ commutes with itself, equivalently, iff $\T \lt \T^\perp$ as subtheories of $\U$.  Hence we deduce the following:
\end{RemSub}

\begin{PropSub}\label{thm:bal_impl_comm_sat}
Any balanced $\J$-theory $\T$ over $\U$ is necessarily a commutative, saturated subtheory of $\U$.
\end{PropSub}

\begin{ExaSub}[\textbf{Maximal commutative subrings as balanced subtheories}]
Let $R$ be a subring of a ring $U$.  Taking $\V = \Ab$, let $\U$ denote the $\{\ZZ\}$-theory corresponding to $U$.  Then the subtheory $\R \hookrightarrow \U$ corresponding to $R$ is balanced if and only if $R$ is equal to its own centralizer $C_U(R)$ in $U$.  It is well-known (and easy to prove) that this is the case if and only if $R$ is a \textit{maximal commutative subring} of $U$, i.e. a maximal element of the poset of commutative subrings of $U$ under inclusion.
\end{ExaSub}

\begin{ExaSub}[\textbf{Double centralizers of left $R$-modules}]
Let $M$ be a left $R$-module for a ring $R$.  Taking $\V = \Ab$ and letting $\R$ denote the $\{\ZZ\}$-theory corresponding to $R$, the $\R$-algebra $M$ determines a morphism of $\{\ZZ\}$-theories $\R \rightarrow \Ab_M$, which is simply the canonical ring homomorphism $R \rightarrow \End_\ZZ(M)$ induced by $M$.  Thus regarding $\R$ as a $\{\ZZ\}$-theory over $\Ab_M$, we deduce by \bref{exa:cmtnt_rmod_over_ab} that the double commutant $\R^{\perp\perp}$ over $\Ab_M$ is precisely the \textit{double centralizer of $M$} in the sense of \cite{DlRi}, i.e. the centralizer of the subring $\End_R(M) \hookrightarrow \End_\ZZ(M)$.  Hence $\R$ is saturated over $\Ab_M$ if and only if the left $R$-module $M$ is faithful and has the \textit{double centralizer property} in the sense of \cite{DlRi}.  The reader is warned that our use of the term \textit{balanced} for $\J$-theories does not accord with the use of this term in ring theory, where it is sometimes used to refer to $R$-modules with the double centralizer property.
\end{ExaSub}

\begin{ExaSub}[\textbf{The opposite ring as a commutant}]\label{exa:crings_bal_subth}
Letting $R$ be a ring and taking $\V = \Ab$, we can regard $R$ as a $\{\ZZ\}$-theory.  The endomorphism ring $\End_\ZZ(R)$ is the full $\{\ZZ\}$-theory $\Ab_R$ of $R$ in $\Ab$.  Since $R$ is a left $R$-module, we have a canonical ring homomorphism $R \rightarrow \End_\ZZ(R)$.  Thus regarding $R$ as a $\{\ZZ\}$-theory over $\End_\ZZ(R)$, the commutant $R^\perp$ of $R$ is the subring $\End_R(R) \hookrightarrow \End_\ZZ(R)$.  On the other hand, since $R$ is also a right $R$-module we have an injective ring homomorphism $R^\op \rightarrow \End_\ZZ(R)$ whose image is precisely $\End_R(R) = R^\perp$, so that $R^\perp \cong R^\op$ as $\{\ZZ\}$-theories over $\End_\ZZ(R)$.  Applying this result also to the ring $R^\op$, we find that $R$ is necessarily saturated when regarded as a $\{\ZZ\}$-theory over $\End_\ZZ(R)$.   Moreover, we claim that $R$ is a balanced $\{\ZZ\}$-theory over $\End_\ZZ(R)$ if and only if $R$ is commutative.  Indeed, if $R$ is a commutative ring then $R = R^\op \cong R^\perp$ as $\{\ZZ\}$-theories over $\End_\ZZ(R)$.  Conversely, if $R$ is a balanced $\{\ZZ\}$-theory over $\End_\ZZ(R)$ then $R$ is a commutative ring by \bref{thm:bal_impl_comm_sat} and \bref{exa:crings_comm_zth}.
\end{ExaSub}

\begin{ExaSub}[\textbf{The Lawvere theories of left and right} $R$-\textbf{modules}]\label{exa:lr_mods_cmtnts}
Any ring or rig $R$ can be viewed as a left $R$-module and so determines a morphism $\Mat_R \rightarrow \Set_R$ from the Lawvere theory of left $R$-modules $\Mat_R$ into the full finitary theory $\Set_R$ of $R$ in $\Set$.  $R$ is also a right $R$-module (equivalently, a left $R^\op$-module) and hence also determines a morphism $\Mat_{R^\op} \rightarrow \Set_R$.  It is proved in \cite[6.5]{Lu:CvxAffCmt} that $\Mat_R$ and $\Mat_{R^\op}$ are commutants of one another over $\Set_R$.  In particular, $\Mat_R$ is a saturated subtheory of $\Set_R$, and this subtheory is balanced if and only if $R$ is commutative \cite[6.5]{Lu:CvxAffCmt}.
\end{ExaSub}

\begin{ExaSub}
By \bref{exa:lr_mods_cmtnts}, the Lawvere theory of join semilattices $\Mat_2$ \pbref{exa:slats} is a balanced subtheory of the Lawvere theory of Boolean algebras $\Set_2$ \pbref{exa:bool}.
\end{ExaSub}

\begin{ExaSub}[\textbf{A non-saturated subtheory}]
Let $k$ be an infinite integral domain, and let $\T$ be the Lawvere theory of commutative $k$-algebras \pbref{exa:lth_ckalgs}.  $k$ itself is a commutative $k$-algebra and so determines a morphism of Lawvere theories $\T \rightarrow \Set_k$ into the full finitary theory $\Set_k$ of $k$ in $\Set$.  This morphism presents $\T$ as a subtheory of $\Set_k$, but this subtheory is not saturated \cite[6.7]{Lu:CvxAffCmt}.  Indeed, $\T^\perp \cong \FinCard^\op$ over $\Set_k$ and consequently $\T^{\perp\perp} \cong \Set_k \not\cong \T$ \cite[6.7]{Lu:CvxAffCmt}.
\end{ExaSub}

\begin{ExaSub}[\textbf{The theories of affine and convex spaces}]
Let $R$ be ring or rig.  By definition, a \textit{pointed right $R$-module} is a right $R$-module $M$ equipped with an arbitrary chosen element $* \in M$.  The category of pointed right $R$-modules (with right $R$-linear maps preserving the chosen points) is isomorphic to the category of normal $\T$-algebras $\Alg{\T}^!$ for a certain Lawvere theory $\T = \Mat_{R^\op}^*$ \cite[7.1]{Lu:CvxAffCmt}.  $R$ itself is a pointed right $R$-module with chosen point $1 \in R$ and so determines a morphism of Lawvere theories $\Mat_{R^\op}^* \rightarrow \Set_R$ into the full finitary theory of $R$ in $\Set$.  Similarly considering the theory of left $R$-affine spaces $\Mat_R^\aff$ \pbref{exa:raff} as a theory over $\Set_R$ via the morphism $\Mat_R^\aff \rightarrow \Set_R$ determined by the left $R$-affine space $R$, it is proved in \cite[7.2]{Lu:CvxAffCmt} that $\Mat_R^\aff$ is the commutant of $\Mat_{R^\op}^*$ over $\Set_R$.  In particular, $\Mat_R^\aff$ is therefore a saturated subtheory of $\Set_R$.  Further, it is proved in \cite[9.3]{Lu:CvxAffCmt} that if $R$ is a \textit{ring} then the theories $\Mat_R^\aff$ and $\Mat_{R^\op}^*$ are commutants of one another over $\Set_R$.  However for rigs $R$ that are not rings this need not hold; for example, when $R$ is the rig $2$ of \bref{exa:unb_slat}, the commutant over $\Set_2$ of the theory $\Mat_2^\aff$ of unbounded join semilattices is the theory of \textit{join semilattices with top element} \cite[8.2]{Lu:CvxAffCmt}.  Nevertheless, for the commutative rig $\RR_+$ of non-negative reals, the theory $\Mat_{\RR_+}^\aff$ of $\RR$-convex spaces \pbref{exa:cvx_sp} and the theory $\Mat_{\RR_+}^*$ of pointed $\RR_+$-modules are commutants of one another over $\Set_{\RR_+}$ \cite[10.20, 10.21]{Lu:CvxAffCmt}.
\end{ExaSub}

\section{The reduction to single-output operations}

By definition, morphisms of theories $A,B$ commute iff $A_{JJ'},B_{KK'}$ commute for all objects $J,J',K,K'$ of $\J$, but we now show that we can fix $J' = I$ and $K' = I$ and still obtain an equivalent condition.

\begin{LemSub}\label{thm:lem_red_single_outp_ops}
Let $A:\T \rightarrow \U$ and $B:\sS \rightarrow \U$ be morphisms of $\J$-theories, and let $K,K' \in \ob\J$.  Then the following are equivalent:
\begin{enumerate}
\item For all $J,J' \in \ob\J$, $A_{JJ'} \perp B_{KK'}$.
\item For all $J \in \ob\J$, $A_{JI} \perp B_{KK'}$.
\end{enumerate}
\end{LemSub}
\begin{proof}
By \bref{thm:lemma_on_cmtn_vs_thoms}, 1 holds if and only if $\theta_{KK'} \cdot B_{KK'}$ is valued in $\T$-homomorphisms from $[K,A]_r$ to $[K',A]_r$, and by \bref{thm:valued_in_thoms_quantifying_over_just_obj} this is equivalent to the statement that for every $J \in \ob\J$, $\theta_{KK'} \cdot B_{KK'}$ preserves $\T$-operations of input arity $J$ and output arity $I$.  But by another application of \bref{thm:lemma_on_cmtn_vs_thoms} this is equivalent to 2.
\end{proof}

\begin{ThmSub}\label{thm:comm_mor_th_via_single_outp_ops}
Let $A:\T \rightarrow \U$ and $B:\sS \rightarrow \U$ be morphisms of $\J$-theories.  Then $A$ and $B$ commute if and only if $A_{JI}$ commutes with $B_{KI}$ for all objects $J,K$ of $\J$.
\end{ThmSub}
\begin{proof}
By \bref{thm:lem_red_single_outp_ops}, $A$ commutes with $B$ if and only if $A_{JI} \perp B_{KK'}$ for all $J,K,K' \in \ob\J$.  By now using the symmetry of $\perp$ and exchanging the roles of $A$ and $B$, we can invoke \bref{thm:lem_red_single_outp_ops} again to deduce that $A$ commutes with $B$ if and only if $B_{KI} \perp A_{JI}$ for all $J,K \in \ob\J$.
\end{proof}

Whereas commutation of morphisms of theories is defined in terms of the Kronecker products $\kk{JJ'KK'}$ and $\kt{JJ'KK'}$, the preceding theorem entails that just the Kronecker products with $J' = I = K'$ suffice, and the form of these can be simplified considerably, as follows.

\begin{DefSub}\label{def:kp_so}
Given a $\J$-theory $\T$ and objects $J,K$ of $\J$, the first and second \textbf{Kronecker products of single-output operations} of arities $J$ and $K$ are defined as
$$\kk{JK} := \left(\T(J,I)\otimes\T(K,I) \xrightarrow{[K,-]_{JI}\otimes 1} \T(J\otimes K,K)\otimes\T(K,I) \xrightarrow{c} \T(J\otimes K,I)\right),$$
$$\kt{JK} := \left(\T(J,I)\otimes\T(K,I) \xrightarrow{1\otimes [J,-]_{KI}} \T(J,I)\otimes\T(J \otimes K,J) \xrightarrow{c} \T(J\otimes K,I)\right),$$
where $c$ denotes the relevant composition morphism, $[K,-]_{JI}$ denotes the morphism induced by the cotensors $[K,J]_r = J\otimes K$ and $[K,I] = K$ per \bref{par:cot}, and $[J,-]_{KI}$ denotes the morphism induced by the cotensors $[J,K]_\ell = J\otimes K$ and $[J,I] = J$.
\end{DefSub}

\begin{PropSub}
Given a $\J$-theory $(\T,\tau)$, the diagram
$$
\xymatrix{
\T(J,I)\otimes\T(K,I) \ar[dr]_{\kk{JK}}\ar[r]^(.53){\kk{JIKI}} & \T(J\otimes K,I\otimes I) \ar[d]^{\T(1,\tau(\ell^{-1}_I))}_\wr\\
& \T(J\otimes K,I)
}
$$
commutes, where the right side is the isomorphism determined by the canonical isomorphism $\ell^{-1}_I = r^{-1}_I:I \rightarrow I\otimes I$ in $\J$.  Further, the similar diagram obtained by substituting $\kt{}$ for $\kk{}$ also commutes.
\end{PropSub}
\begin{proof}
Observe that the given diagram is the same as the periphery of the following diagram
$$
\xymatrix{
\T(J,I)\T(K,I) \ar[drr]_{[K,-]\otimes 1} \ar[rr]^(.45){[K,-]_r[I,-]_\ell} & & \T(JK,IK)\T(IK,II) \ar[d]|{\T(1,\tau(\ell^{-1}))\T(\tau(\ell),\tau(\ell^{-1}))} \ar[r]^(.6)c & \T(JK,II) \ar[d]^{\T(1,\tau(\ell^{-1}))}\\
& & \T(JK,K)\T(K,I) \ar[r]_(.6)c & \T(JK,I)
}
$$
which commutes, since the rightmost square clearly commutes and the commutativity of the leftmost square follows from the following claims:
\begin{enumerate}
\item $[I,-]_\ell:\T(K,I) \rightarrow \T(IK,II)$ is equal to $\T(\tau(\ell^{-1}_K),\tau(\ell_I))$.
\item The following diagram commutes.
$$
\xymatrix{
\T(J,I) \ar[dr]_{[K,-]} \ar[r]^(.45){[K,-]_r} & \T(JK,IK) \ar[d]^{\T(1,\tau(\ell^{-1}_K))}\\
& \T(JK,K)
}
$$
\end{enumerate}

In order to prove 1, observe that we have two cotensors $[I,K] = K$ and $[I,K]_\ell = IK$ of the same object $K$ by $I$ in $\T$, and we claim that the induced isomorphism $[I,K] \rightarrow [I,K]_\ell$ is simply $\tau(\ell_K):K \rightarrow IK$.  Indeed, the counit for the cotensor $[I,K]_\ell = IK$ is defined as the composite
\begin{equation}\label{eq:counit_left_cot_by_i}I \xrightarrow{\Coev} \uV(K,IK) = \J^\op(IK,K) \xrightarrow{\tau_{IK,K}} \T(IK,K)\;,\end{equation}
but one readily verifies that the coevaluation morphism $\Coev$ here is simply the morphism $[\ell^{-1}_{K}]$ that picks out the canonical isomorphism $\ell_K^{-1}:K \rightarrow IK$.  Hence the counit \eqref{eq:counit_left_cot_by_i} for $[I,K]_\ell$ is $[\tau(\ell_K^{-1})]$, whereas the counit for $[I,K] = K$ is the identity arrow $[1_K]:I \rightarrow \T(K,K)$, so the morphism $\T(\tau(\ell_K),K):\T(IK,K) \rightarrow \T(K,K)$ commutes with these cotensor counits, proving that $\tau(\ell_K)$ is the induced isomorphism of cotensors, as needed.  Similarly, we have two cotensors $[I,I] = I$ and $[I,I]_\ell = II$ of $I$ by $I$ in $\T$, and, by the same reasoning, the induced isomorphism $[I,I] \rightarrow [I,I]_\ell$ is $\tau(\ell_I)$.  We can now invoke \bref{thm:cots_of_same_obj} to deduce that 1 holds, using the fact that the morphism $[I,-]:\T(K,I) \rightarrow \T(K,I)$ induced by the cotensors $[I,K] = K$ and $[I,I] = I$ is the identity morphism.

To prove 2, note that we have a pair of cotensors $[K,I]_r = IK$ and $[K,I] = K$ of the same object $I$ of $\T$ by the object $K$ of $\V$, and we claim that the induced isomorphism $[K,I]_r \rightarrow [K,I]$ is $\tau(\ell^{-1}_K):IK \rightarrow K$.  Indeed, for this it suffices to show that the following diagram commutes
$$
\xymatrix{
K \ar[d]_{\gamma_K} \ar[r]^{\gamma'^I_K} & \T(IK,I)\\
\T(K,I) \ar[ur]_{\T(\tau(\ell^{-1}_K),1)} & 
}
$$
where $\gamma_K$ and $\gamma'^I_K$ denote the respective cotensor counits, and this follows readily from the definition of $\gamma_K$ and the characterization of $\gamma'^I_K$ given at \eqref{eq:direct_charn_right_cot}.  Hence we can now invoke \bref{thm:cots_of_same_obj} with $\C = \T$, $V = K$, $D_1 = J$, $D_2 = I$, $[V,D_1]^0 = [K,J]_r = [V,D_1]^1$, $[V,D_2]^0 = [K,I]_r$, and $[V,D_2]^1 = K$ to deduce that 2 holds.
\end{proof}

\begin{CorSub}
Let $\T$ be $\J$-theory, let $J,K$ be objects of $\J$, and let $\mu:V \rightarrow \T(J,I)$ and $\nu:W \rightarrow \T(K,I)$ be morphisms in $\V$.  Then $\mu$ commutes with $\nu$ if and only if
$$\kk{JK} \cdot (\mu \otimes \nu) = \kt{JK} \cdot (\mu \otimes \nu)\;.$$
\end{CorSub}

This immediately entails the following corollary to \bref{thm:comm_mor_th_via_single_outp_ops}:

\begin{ThmSub}\label{thm:cmtn_via_single_outp_ops}
Morphisms of $\J$-theories $A:\T \rightarrow \U$ and $B:\sS \rightarrow \U$ commute if and only if $\kk{JK} \cdot (A_{JI} \otimes B_{KI}) = \kt{JK} \cdot (A_{JI} \otimes B_{KI})$ for all objects $J,K$ of $\J$.
\end{ThmSub}

\section{Commutants for \texorpdfstring{\kern -0.5ex $\J$}{J}-ary monads on \texorpdfstring{$\V$}{V}}

\begin{ParSub}[\textbf{Correspondence between} $\J$-\textbf{theories and} $\J$-\textbf{ary monads}]\label{par:equiv_jth_jary_mnds}
Given a system of arities $j:\J \hookrightarrow \uV$, we say that a $\V$-monad $\TT = (T,\eta,\mu)$ on $\uV$ is a $\J$-\textbf{ary} $\V$-\textbf{monad} \cite[\S 11]{Lu:EnrAlgTh} if $T$ preserves ($\V$-enriched) left Kan extensions along $j$.  For example, for the system of arities $\J = \FinCard \hookrightarrow \Set = \V$, we recover the usual notion of \textit{finitary monad} \cite[11.3]{Lu:EnrAlgTh}.  It is shown in \cite[\S 11]{Lu:EnrAlgTh} that there is an equivalence between $\J$-theories and $\J$-ary $\V$-monads on $\uV$ \cite[11.8]{Lu:EnrAlgTh} as soon as the system of arities $j:\J \hookrightarrow \uV$ is \textbf{eleutheric} \cite[\S 7]{Lu:EnrAlgTh}.  The latter condition on $j$ means that every $\V$-functor $\J \rightarrow \uV$ has a left Kan extension along $j$ and that, furthermore, these Kan extensions are preserved by the $\V$-functors $\uV(J,-):\uV \rightarrow \uV$ associated to objects $J$ of $\J$.  Each of the systems of arities listed in Example \bref{exa:sys_ar_jth}(a)-(d) is eleutheric \cite[7.5]{Lu:EnrAlgTh}, and the system of arities in \bref{exa:sys_ar_jth}(e) is eleutheric for a broad class of categories $\V$ \cite[7.5 \#5]{Lu:EnrAlgTh} that includes every countably cocomplete cartesian closed category $\V$.  For the remainder of this section we shall fix an eleutheric system of arities $j:\J \hookrightarrow \uV$.  The precise result relating $\J$-theories and $\J$-ary monads is then as follows:
\end{ParSub}

\begin{ThmSub}[{\cite[11.8]{Lu:EnrAlgTh}}]\label{thm:equiv_jth_jmnd}
There is an equivalence
$$\ThJ \;\;\;\;\simeq\;\;\;\; \MndJ(\uV)$$
between the category $\ThJ$ of $\J$-theories and the full subcategory $\MndJ(\uV)$ of the category of $\V$-monads on $\uV$ with objects all $\J$-ary $\V$-monads.
\end{ThmSub}

\begin{ParSub}\label{par:passage_betw_th_mnds}
Explicitly, given a $\J$-theory $\T$ one obtains a $\V$-monad $\TT = \mathsf{m}(\T)$ whose underlying endo-$\V$-functor $T:\uV \rightarrow \uV$ is the left Kan extension of
\begin{equation}\label{eqn_tsubi}\T_I := \T(\tau-,I):\J \rightarrow \uV\end{equation}
along $j:\J \hookrightarrow \uV$, where $\tau:\J^\op \rightarrow \T$ is the identity-on-objects $\V$-functor associated to $\T$.  Given a morphism $A:\T \rightarrow \U$ between $\J$-theories $(\T,\tau)$ and $(\U,\upsilon)$, the associated morphism $\mathsf{m}(A):\mathsf{m}(\T) \rightarrow \mathsf{m}(\U)$ is obtained by applying $\Lan_j:\VCAT(\J,\uV) \rightarrow \VCAT(\uV,\uV)$ to the $\V$-natural transformation
$A_{\tau-,I}:\T(\tau-,I) \rightarrow \U(A\tau-,AI) = \U(\upsilon-,I)\;,$
recalling that $A \circ \tau = \upsilon$ since $A$ is a morphism of $\J$-theories.

In the other direction, given a $\J$-ary $\V$-monad $\TT$ on $\uV$, let $\uV_\TT$ denote the Kleisli $\V$-category for $\TT$ and let $\J_\TT$ denote its full sub-$\V$-category on the objects of $\J$.  The $\J$-theory $\mathsf{t}(\TT)$ associated to $\TT$ is then the opposite $\J_\TT^\op$, which we therefore call the \textbf{Kleisli} $\J$-\textbf{theory} for $\TT$.  These assignments extend to mutually pseudo-inverse functors $\mathsf{m}$, $\mathsf{t}$ between $\ThJ$ and $\MndJ(\uV)$.

In particular, if we take $\J = \uV$ and $j = 1_{\uV}$ then \bref{thm:equiv_jth_jmnd} yields an equivalence $\Th_{\uV} \simeq \Mnd_{\uV}(\uV) = \Mnd_{\VCAT}(\uV)$ between $\uV$-theories and arbitrary $\V$-monads on $\uV$, since each of the latter is $\uV$-ary, trivially.
\end{ParSub}

A notion of commutation of morphisms of arbitrary $\V$-monads on $\uV$ was introduced by Kock in the paper \cite{Kock:DblDln} of 1970, and we shall now reconcile that notion with the notion of commutation of morphisms of $\J$-theories.  Kock had defined the notion of \textit{commutative monad} in \cite{Kock:Comm}, observing that for any $\V$-monad $\TT = (T,\eta,\mu)$ on $\uV$ one can define for each pair of objects $V,W$ of $\V$ a pair of canonical morphisms
$$\kappa^\TT_{VW},\:\widetilde{\kappa}^\TT_{VW} \colon TV \otimes TW \rightarrow T(V \otimes W)$$
(see \cite[2.1, 3.1]{Kock:Comm}) that we shall call the first and second \textbf{Kock-Kronecker products} carried by $\TT$.  One says that $\TT$ is a \textbf{commutative monad} if $\kappa^\TT_{VW} = \widetilde{\kappa}^\TT_{VW}$ for all objects $V$ and $W$.  Kock's notion of commutation generalizes this:

\begin{DefSub}[{Kock, \cite[4.1]{Kock:DblDln}}]
Let $\alpha:\TT \rightarrow \UU$ and $\beta:\SSS \rightarrow \UU$ be morphisms of $\V$-monads on $\uV$.  We say that $\alpha$ \textbf{commutes with} $\beta$ if the two composites in
$$\xymatrix{TV \otimes SW \ar[rr]^{\alpha_V \otimes \beta_W} & & UV \otimes UW \ar@<.5ex>[rr]^{\kappa^\UU_{VW}} \ar@<-.5ex>[rr]_{\widetilde{\kappa}^\UU_{VW}} & & U(V \otimes W)}$$
are equal for all objects $V$ and $W$ of $\V$.
\end{DefSub}

\begin{ThmSub}\label{thm:cmmtn_jth_jmnd}
Let $A:\T \rightarrow \U$ and $B:\sS \rightarrow \U$ be morphisms of $\J$-theories, and let $\alpha:\TT \rightarrow \UU$ and $\beta:\SSS \rightarrow \UU$ denote the corresponding morphisms of $\J$-ary $\V$-monads on $\uV$.  Then $A$ commutes with $B$ if and only if $\alpha$ commutes with $\beta$.
\end{ThmSub}
\begin{proof}
The morphisms $\kappa^\UU_{VW},\:\widetilde{\kappa}^\UU_{VW}$ constitute $\V$-natural transformations $\kappa,\widetilde{\kappa}$ as in the leftmost of the following diagrams.
\begin{equation}\label{eqn:pasting_diagrams_for_cmtn}
\xymatrix{
\uV\otimes\uV \ar@/_3ex/[d]_{T \otimes S}|{}="s3" \ar@/^3ex/[d]^{U \otimes U}|{}="t3" \ar[rr]^\otimes & & \uV \ar@/^3ex/[d]^U \ar@{}[dll]|(.45){}="s2"|(.3){}="t2" & & & \J\otimes\J \ar@/_3ex/[d]_{\T_I\otimes\sS_I}|{}="s4" \ar@/^3ex/[d]^{\U_I\otimes\U_I}|{}="t4" \ar[rr]^\otimes & & \J \ar@/^3ex/[d]^{\U_I} \ar@{}[dll]|(.4){}="s1"|(.28){}="t1" \\
\uV\otimes\uV \ar[rr]_\otimes & & \uV & & & \uV\otimes\uV \ar[rr]_\otimes & & \uV
\ar@{=>}"s1";"t1"_{\mathsf{k},\:\widetilde{\mathsf{k}}}
\ar@{=>}"s2";"t2"_{\kappa,\:\widetilde{\kappa}}
\ar@{}"s3";"t3"|(.2){}="s5"|(.8){}="t5"
\ar@{}"s4";"t4"|(.2){}="s6"|(.8){}="t6"
\ar@{=>}"s5";"t5"^{\alpha \otimes \beta}
\ar@{=>}"s6";"t6"^{A \otimes B}
}
\end{equation}
The first and second single-output Kronecker products $\mathsf{k}$ and $\widetilde{\mathsf{k}}$ for $\U$ \pbref{def:kp_so} constitute $\V$-natural transformations as in the rightmost diagram, where we have employed the notation $\T_I = \T(\tau-,I)$ of \eqref{eqn_tsubi} and written simply $A$ for the natural transformation $A_{\tau-,I}:\T_I \rightarrow \U_I$ of \bref{par:passage_betw_th_mnds}, and similarly for $B$.

Now $\alpha$ commutes with $\beta$ iff the leftmost diagram is a \textit{fork}, meaning that the pasted 2-cells involving $\kappa,\widetilde{\kappa}$ obtained therein are equal, whereas $A$ commutes with $B$ iff the rightmost diagram is a fork \pbref{thm:cmtn_via_single_outp_ops}.  Since $T = \Lan_j\T_I$ and $S = \Lan_j\sS_I$, it follows by a short computation with coends that the composite $\uV\otimes\uV \xrightarrow{T\otimes S} \uV\otimes\uV \xrightarrow{\otimes} \uV$ is a left Kan extension of its restriction along $j\otimes j:\J\otimes\J \rightarrow \uV\otimes\uV$.  From this it follows by \cite[4.43]{Ke:Ba} that the leftmost diagram in \eqref{eqn:pasting_diagrams_for_cmtn} is a fork iff it `is a fork when whiskered with $j\otimes j$', i.e. iff
\begin{equation}\label{eq:eqn_for_left_diag}\kappa \circ (\alpha \otimes \beta) \circ (j \otimes j) = \widetilde{\kappa} \circ (\alpha \otimes \beta) \circ (j \otimes j)\;,\end{equation}
where $\circ$ denotes pasting/whiskering as applicable.  Hence it is our task to show that the latter equation is equivalent to the statement that the rightmost diagram in \eqref{eqn:pasting_diagrams_for_cmtn} is a fork.

In the diagram
\begin{equation}\label{eq:prism}
\xymatrix{
\J\otimes\J \ar[dd]_{\T_I\otimes\sS_I} \ar@{=}[rr] \ar[dr]^{j\otimes j} & & \J\otimes\J \ar@{..>}[dd]|(.3){\U_I\otimes\U_I} \ar[dr]^{j\otimes j} \ar[rr]^\otimes & & \J \ar@{..>}[dd]|(.3){\U_I} \ar[dr]^j\\
& \uV\otimes\uV \ar[dl]^{T\otimes S} \ar@{=}[rr] & & \uV\otimes\uV \ar[dl]^{U\otimes U} \ar[rr]^(.4)\otimes & & \uV \ar[dl]^U\\
\uV\otimes\uV \ar@{=}[rr] & & \uV\otimes\uV \ar[rr]_\otimes & & \uV
}
\end{equation}
let the 2-cell $\alpha \otimes \beta$ occupy the leftmost cell on the lower front face of the triangular prism (which we visualize as protruding from the page with the dashed lines behind the prism).  Let the 2-cell $A \otimes B$ occupy the left cell on the back face.  Let the 2-cell $\mathsf{k}$ occupy the rightmost cell on the back face, and let $\kappa$ occupy the rightmost cell on the lower front face.  Observe that the cells on the upper front face commute strictly.  Since $U = \Lan_j\U_I$, we have a canonical $\V$-natural transformation $\theta^{\U}:\U_I \Rightarrow U \circ j$, namely the component at $\U_I$ of the unit of the left Kan extension adjunction $\Lan_j \dashv (-) \circ j:\VCAT(\uV,\uV) \rightarrow \VCAT(\J,\uV)$, and since $j$ is fully faithful, $\theta^{\U}$ is an invertible 2-cell that occupies the rightmost face of the prism.  We therefore also have an invertible 2-cell $\theta^\U\otimes \theta^\U$ that occupies the triangular cell within the interior of the prism, and similarly we also have an invertible 2-cell $\theta^{\T} \otimes \theta^{\sS}$ that occupies the leftmost face.

Since the 2-cells on the left and right faces of the prism \eqref{eq:prism} are invertible, we can reason that it now suffices to show that the surface of the prism \eqref{eq:prism} `commutes' (in the sense that the 2-cell that results from pasting its lower front, upper front, and left faces is equal to the 2-cell obtained by pasting its back and right faces) and that the analogous prism with $\widetilde{\mathsf{k}},\widetilde{\kappa}$ in place of $\mathsf{k},\kappa$ commutes as well.  We prove the first of these claims; the second is then established similarly.  To this end, first observe that the 3-dimensional cell constituting the left half of the prism commutes in the given sense, since by definition $\alpha$ and $\beta$ are the images of $A$ and $B$ under the left adjoint $\Lan_j:\VCAT(\J,\uV) \rightarrow \VCAT(\uV,\uV)$.  We claim that the rightmost half of the prism also commutes.  To show this, we must prove that for each pair of objects $J,K$ of $\J$ the diagram
\begin{equation}\label{eq:kronecker_products_commute_w_isos1}
\xymatrix{
\U_IJ\otimes\U_IK \ar[d]_{\mathsf{k}_{JK}} \ar[rr]^{\theta^\U_J\otimes\theta^\U_K} & & UJ \otimes UK \ar[d]^{\kappa_{JK}}\\
\U_I(J\otimes K) \ar[rr]_{\theta^\U_{J\otimes K}} & & U(J\otimes K)
}
\end{equation}
commutes.  To this end, note that since $\UU = \mathsf{m}(\U)$ was obtained from $\U$ via the equivalence $\ThJ \simeq \MndJ(\uV)$, we have an isomorphism $\eta^\U:\U \xrightarrow{\sim} \mathsf{t}(\mathsf{m}(\U)) = \J_\UU^\op$, recalling that $\J_\UU^\op$ denotes the Kleisli $\J$-theory \pbref{par:passage_betw_th_mnds}.  Since $\eta^\U$ is a morphism of $\J$-theories, it follows by \bref{def:morph_th}, \bref{rem:morph_pres_lr_jcots}, \bref{def:kp_so} that the leftmost square in
\begin{equation}\label{eq:kronecker_products_commute_w_isos2}
\xymatrix{
\U(J,I)\otimes\U(K,I) \ar[d]_{\mathsf{k}_{JK}} \ar[rr]^{\eta^\U_{JI}\otimes\eta^\U_{KI}} & & \J^\op_\UU(J,I)\otimes\J_\UU^\op(K,I) \ar[d]^{\mathsf{k}_{JK}} \ar[rr]^\sim & & UJ\otimes UK \ar[d]^{\kappa_{JK}}\\
\U(J \otimes K,I) \ar[rr]_{\eta^\U_{J\otimes K,I}} & & \J_\UU^\op(J\otimes K,I) \ar[rr]_\sim & & U(J\otimes K)
}
\end{equation}
commutes.  The horizontal arrows in the rightmost square are obtained from the canonical isomorphisms $\J_\UU^\op(L,I) = \uV(I,UL) \cong UL$ for objects $L$ of $\J$, and by the definition of the equivalence $\ThJ \simeq \MndJ(\uV)$ in \cite[11.8, 11.6]{Lu:EnrAlgTh} we have that $\eta^\U_{LI}:\U(L,I) \rightarrow \J^\op_\UU(L,I)$ is the composite $\xymatrix{\U(L,I) \ar[rr]|(.55){\theta^\U_L} & & UL \ar[r]^(.4)\sim & \J^\op_\UU(L,I)}$ whose second factor is this canonical isomorphism.  Hence the periphery of \eqref{eq:kronecker_products_commute_w_isos2} is the square \eqref{eq:kronecker_products_commute_w_isos1}, which therefore commutes as soon as we can show that the rightmost square in \eqref{eq:kronecker_products_commute_w_isos2} commutes.  But this follows from \cite[6.2.5]{Lu:PhD}, wherein it is proved by elementary means that the analogous square with $\uV_\UU^\op$ in place of $\J_\UU^\op$ commutes for any pair of objects of $\V$ in place of $J,K$, and for any $\V$-monad $\UU$ on $\uV$.
\end{proof}

\begin{CorSub}\label{thm:jary_mnd_comm_iff_jth_comm}
A $\J$-ary monad $\TT$ is commutative if and only if its corresponding $\J$-theory is commutative.
\end{CorSub}

\begin{RemSub}
When applying \bref{thm:cmmtn_jth_jmnd} and \bref{thm:jary_mnd_comm_iff_jth_comm} it is important to know that the notion of commutation of cospans of $\J$-theories (resp. $\V$-monads) is invariant under isomorphism of cospans (considered as diagrams of shape $\cdot \rightarrow \cdot \leftarrow \cdot$).  This is readily verified using \bref{thm:basic_props_of_cmmtn_morphs} and a similar proposition for $\V$-monads \cite[4.3]{Kock:DblDln}. 
\end{RemSub}

\begin{DefSub}\label{def:jary_and_abs_cmtnt_for_mnds}
Let $\alpha:\TT \rightarrow \UU$ be a morphism of $\V$-monads on $\uV$.
\begin{enumerate}
\item If $\TT$ and $\UU$ are $\J$-ary $\V$-monads, then we define the $\J$-\textbf{ary} \textbf{commutant} of $\alpha$ (or of $\TT$ with respect to $\alpha$) to be the $\J$-ary $\V$-monad $\TT^{\perp}_{\alpha,j}$ associated to the commutant $(\mathsf{t}(\TT))^\perp_{\mathsf{t}(\alpha)}$ of the morphism of $\J$-theories $\mathsf{t}(\alpha):\mathsf{t}(\TT) \rightarrow \mathsf{t}(\UU)$ associated to $\alpha$, \textit{provided that the latter commutant exists}.
\item We define the \textbf{(absolute) commutant} $\TT^\perp_\alpha$ of $\TT$ with respect to $\alpha$ to be the $\uV$-ary commutant of $\TT$ with respect to $\alpha$, \textit{provided that the latter commutant exists}.
\end{enumerate}
\end{DefSub}

\begin{RemSub}\label{rem:existence_of_cmtnt_for_mnds}
By \bref{thm:existence_of_commutant_via_intersections}, if $\V$ has intersections of $(\ob\J)$-indexed families of strong subobjects, then the $\J$-ary commutant always exists.  In particular, if $\V$ is complete and well-powered with respect to strong subobjects, then the absolute commutant exists for any morphism of $\V$-monads on $\uV$.
\end{RemSub}

\begin{RemSub}
Since we have an equivalence $\ThJ \simeq \MndJ(\uV)$ and the notions of commutation in these two categories agree \pbref{thm:cmmtn_jth_jmnd}, several of our results and definitions concerning commutants and commutation for $\J$-theories can be transposed to the setting of $\J$-ary monads, and with $\J = \uV$ they apply also to the absolute commutant for arbitrary $\V$-monads on $\uV$.  In particular, we deduce by \bref{thm:commutants_via_commutativity} and \bref{thm:cmmtn_jth_jmnd} that the $\J$-ary commutant is characterized by a universal property when it exists: 
\end{RemSub}

\begin{ThmSub}\label{thm:univ_prop_jary_cmmtnt}
Let $\alpha:\TT \rightarrow \UU$ and $\beta:\SSS \rightarrow \UU$ be morphisms of $\J$-ary monads on $\uV$, and suppose that the $\J$-ary commutant of $\alpha$ exists.  Then $\alpha$ and $\beta$ commute if and only if $\beta$ factors through the $\J$-ary commutant $\TT^\perp_{\alpha,j} \rightarrow \UU$ of $\alpha$.
\end{ThmSub}

\begin{RemSub}\label{rem:jary_vs_absolute}
The factorization of $\beta$ through the $\J$-ary commutant in \bref{thm:univ_prop_jary_cmmtnt} is unique if it exists, as $\TT^\perp_{\alpha,j} \rightarrow \UU$ is a monomorphism in $\MndJ(\uV) \simeq \ThJ$ since its corresponding morphism of $\J$-theories $\T^\perp \hookrightarrow \U$ is a subtheory inclusion.  But beware---we have no reason to expect in general that the $\J$-ary commutant $\TT^\perp_{\alpha,j}$ would be a \textit{submonad} of $\UU$, as the morphism $T^\perp_{\alpha,j} \rightarrow U$ is obtained from the inclusion $\T^\perp_I \hookrightarrow \U_I$ (in the notation of \bref{eqn_tsubi}) by applying the left Kan extension functor $\Lan_j:\VCAT(\J,\uV) \rightarrow \VCAT(\uV,\uV)$, which need not preserve monomorphisms in general.  Indeed, consider the case $\V = \Ab$, $j:\J = \{\ZZ\} \hookrightarrow \Ab$, where $\VCAT(\J,\uV) \cong \Ab$ and $\Lan_j$ sends an abelian group $M$ to the additive endofunctor $M \otimes (-)$ on $\Ab$.

Hence we have no reason to expect that the $\J$-ary commutant of a morphism of $\J$-ary monads would in general coincide with its absolute commutant, whose canonical morphism $\TT^\perp \hookrightarrow \UU$ \textit{is} always a submonad inclusion, its components being simply the components $T^\perp V = (\uV^\op_\TT)^\perp(V,I) \hookrightarrow \uV_\UU^\op(V,I) \cong UV$ of the corresponding inclusion of $\uV$-theories $(\uV^\op_\TT)^\perp \hookrightarrow \uV^\op_\UU$.

However, there is one important special case in which the $\J$-ary commutant coincides with the absolute commutant, as follows.  Take $\V = \Set$ and $j:\J = \FinCard \hookrightarrow \Set$, so that $\J$-theories are now the classical Lawvere theories and $\J$-ary monads are the familiar finitary monads on $\Set$.  Here the left Kan extension functor $\Lan_j:\CAT(\FinCard,\Set) \rightarrow \CAT(\Set,\Set)$ \textit{does} preserve monomorphisms, since the left Kan extension $\Lan_j P$ of a functor $P:\FinCard \rightarrow \Set$ is given pointwise as a filtered colimit, and pullbacks commute with filtered colimits in $\Set$.  Moreover, further special properties of $\Set$ allow us to prove the following result, wherein we call the $\J$-ary commutant for $\J = \FinCard$ the \textit{finitary commutant}:
\end{RemSub}

\begin{ThmSub}\label{thm:abs_cmtnt_is_finitary_cmtnt_over_set}
Let $\alpha:\TT \rightarrow \UU$ be a morphism of finitary monads on $\Set$.  Then the finitary commutant of $\alpha$ is the same as the absolute commutant $\TT^\perp$ of $\alpha$.  In particular, the absolute commutant of $\alpha$ is a finitary monad.
\end{ThmSub}
\begin{proof}
$\TT$ and $\UU$ are isomorphic to the finitary monads associated to Lawvere theories $\T$ and $\U$, so w.l.o.g. $\TT = \mathsf{m}(\T)$, $\UU = \mathsf{m}(\U)$, and $\alpha$ is induced by a morphism of Lawvere theories $A:\T \rightarrow \U$.  The finitary commutant $\TT^\perp_j$ of $\alpha$ is the finitary monad associated to the commutant $\T^\perp$ of $A$, and the associated morphism $\varphi:\TT^\perp_j \rightarrow \UU$ is induced by the inclusion of Lawvere theories $\T^\perp \hookrightarrow \U$.  $\varphi$ commutes with $\alpha$ and so factors through the absolute commutant $\TT^\perp \hookrightarrow \UU$ of $\alpha$ via a unique morphism $\varphi':\TT^\perp_j \rightarrow \TT^\perp$, and it suffices to show that the component $\varphi'_X:T^\perp_jX \rightarrow T^\perp X$ is bijective for each set $X$.  But by the preceding remarks $\varphi_X:T^\perp_jX \rightarrow UX$ is injective, so $\varphi'_X$ is injective and it suffices to show that $\varphi'_X$ is surjective.

For each finite cardinal $n$, we shall write $S^n$ to denote $n$ when considered as an object of the Lawvere theory $\U$, so that $S^n$ is an $n$-th power of $S = S^1$ in $\U$, and we shall use the same notation for the subtheory $\T^\perp \hookrightarrow \U$.  Thus we write $S^{(-)}:\FinCard^\op \rightarrow \U$ and $S^{(-)}:\FinCard^\op \rightarrow \T^\perp$ for the unique morphisms of Lawvere theories.  The endofunctors $U,T^\perp_j$ are then the left Kan extensions along $j:\FinCard \hookrightarrow \Set$ of $\U(S^{(-)},S),\T^\perp(S^{(-)},S):\FinCard \rightarrow \Set$, respectively.  Hence the sets $UX$ and $T^\perp_jX$ are the filtered colimits
$$UX \;\;= \varinjlim_{x\::\:n\: \rightarrow\: X}\U(S^n,S)\;\;\;\;\;\;\;\;\;\;\;\;\;\;T^\perp_jX\;\; = \varinjlim_{x\::\:n\: \rightarrow\: X}\T^\perp(S^n,S)\;,$$
taken over the comma category $\FinCard \slash X = (j \downarrow X)$.  The elements of $UX$ are therefore equivalence classes $[\mu,n,x]$ of triples consisting of a finite cardinal $n$, a function $x:n \rightarrow X$, and an abstract operation $\mu:S^n \rightarrow S$ in $\U$, where $[\mu,n,x] = [\nu,m,y]$ iff there exist a finite cardinal $k$ and maps $z:k \rightarrow X$, $f:n \rightarrow k$, $g:m \rightarrow k$ in $\Set$ such that $z \cdot f = x:n \rightarrow X$, $z \cdot g = y:m \rightarrow X$, and $\mu \cdot S^f = \nu \cdot  S^g:S^k \rightarrow S$.  Since the canonical map $T^\perp_jX \rightarrow UX$ is injective, we can identify $T^\perp_jX$ with the subset of $UX$ consisting of the elements that can be represented in the form $[\mu,n,x]$ with $\mu \in \T^\perp(S^n,S) \subseteq \U(S^n,S)$.

Every element of $UX$ can be represented as $[\mu,n,x]$ with $x:n \rightarrow X$ injective, since given arbitrary $x:n \rightarrow X$ and $\mu:S^n \rightarrow S$ in $\U$ we can factor $x$ as a surjection $f:n \rightarrow n'$ followed by an injection $x':n' \rightarrow X$, and then $[\mu,n,x] = [\mu \cdot S^f,n',x']$.

Related to this, we shall require the following:
\begin{description}
\item[\textnormal{\textit{Claim.}}] \textit{Suppose that $[\mu,n,x] = [\nu,n,x]$ in $UX$ with $x:n \rightarrow X$ injective.  Then $\mu = \nu$.}
\end{description}
To prove this, note that the hypothesis entails that there exist $f,g:n \rightarrow k$ in $\FinCard$ and $z:k \rightarrow X$ in $\Set$ with $z \cdot f = x = z \cdot g$ and $\mu \cdot S^f = \nu \cdot S^g$.  Forming the coequalizer $q:k \rightarrow \ell$ of $f,g$ in $\FinCard$, which is also a coequalizer in $\Set$, there is an induced $z':\ell \rightarrow X$ with $z' \cdot q = z$, and then letting $h = q \cdot f = q \cdot g:n \rightarrow \ell$ we have that $z' \cdot h = x:n \rightarrow X$ and $\mu \cdot S^h = \nu \cdot S^h:S^\ell \rightarrow S$.  In order to show that $\mu = \nu$ it suffices to show that $C(\mu) = C(\nu):\ca{C}^n \rightarrow \ca{C}$ for any normal $\U$-algebra $C:\U \rightarrow \Set$.  But we know that $C(\mu) \cdot \ca{C}^h = C(\nu) \cdot \ca{C}^h:\ca{C}^\ell \rightarrow \ca{C}$, where $\ca{C}^h:\ca{C}^\ell \rightarrow \ca{C}^n$ is the map induced by $h$, and $h$ is injective since $x = h \cdot z'$ is injective.  It follows that $\ca{C}^h$ is surjective if $\ca{C} \neq \emptyset$, so that then $C(\mu) = C(\nu)$ as needed, but on the other hand if $\ca{C} = \emptyset$ then $C(\mu),C(\nu):\emptyset^n \rightarrow \emptyset$ and so $\emptyset^n = \emptyset$ (equivalently $n \neq 0$) and again $C(\mu) = C(\nu)$.

Now let $\omega$ be an element of the subset $T^\perp X \hookrightarrow UX$.  Then $\omega$ is of the form $\omega = [\mu,n,x]$ with $x$ injective, and it suffices to show that the element $\mu \in \U(S^n,S)$ lies in $\T^\perp(S^n,S)$, for then $\omega$ lies in the subset $T^\perp_jX \hookrightarrow UX$.  Letting $\nu \in \U(S^m,S)$ lie in the image of $A:\T \rightarrow \U$, we must show that $\mu$ commutes with $\nu$.  But we know that $\omega = [\mu,n,x]$ commutes with every element $\sigma \in UY$ of the form $\sigma = [\nu,m,y]$ for any set $Y$ and any map $y:m \rightarrow Y$, i.e. the maps $\kappa^\UU_{XY},\widetilde{\kappa}^\UU_{XY}:UX \times UY \rightarrow U(X \times Y)$ yield the same value on the pair $(\omega,\sigma)$.  But $\kappa^\UU_{XY}$ sends $(\omega,\sigma)$ to the equivalence class $[\mu * \nu,n \times m,x \times y] \in U(X \times Y)$ of the first Kronecker product $\mu * \nu \in \U(S^{n \times m},S)$ for the map $x \times y:n \times m \rightarrow X \times Y$, and analogously for $\widetilde{\kappa}^\UU_{X \times Y}$ and the second Kronecker product $\mu \stt \nu$, so $[\mu * \nu,n \times m,x \times y] = [\mu \stt \nu,n \times m,x \times y]$.  In particular, we can take $Y = m$ and $y = 1:m \rightarrow m$, whence $[\mu * \nu,n \times m,x \times 1] = [\mu \stt \nu,n \times m,x \times 1]$ as elements of $U(X \times m)$.  But $x \times 1:n \times m \rightarrow X \times m$ is injective since $x$ is so, and therefore $\mu * \nu = \mu \stt \nu$ by the preceding Claim, so $\mu$ commutes with $\nu$.
\end{proof}

\begin{ParSub}
Let $\T$ be a $\J$-theory for the given eleutheric system of arities $\J \hookrightarrow \uV$, and assume that $\V$ has equalizers.  It is shown in \cite[11.14]{Lu:EnrAlgTh} that the $\V$-category $\Alg{\T}$ of $\T$-algebras in $\uV$ always exists and is equivalent to the $\V$-category $\uV^\TT$ of $\TT$-algebras for the associated $\J$-ary $\V$-monad $\TT = \mathsf{m}(\T)$.  Further, the full sub-$\V$-category $\Alg{\T}^! \hookrightarrow \Alg{\T}$ consisting of normal $\T$-algebras is \textit{isomorphic} to $\uV^\TT$ \cite[11.14]{Lu:EnrAlgTh}.
\end{ParSub}

\begin{ThmSub}\label{thm:cmtnt_alg_ele}
Let $A:\T \rightarrow \uV$ be a $\T$-algebra for a $\J$-theory $\T$.  Then the commutant $\T^\perp_A \hookrightarrow \uV_A$ of $A$ exists, recalling that $\uV_A$ is the full $\J$-theory of $A$ in $\uV$ \pbref{def:full_theory}.
\end{ThmSub}
\begin{proof}
By the preceding remark, $\Alg{\T}$ exists, and $\T^\perp_A$ is equivalently defined as the full $\J$-theory of $A$ in $\Alg{\T}$ \pbref{def:cmtnt_talg}.
\end{proof}

\begin{DefSub}\label{def:abs_cmtnt_talg}
Let $\V$ be a symmetric monoidal closed category with equalizers, let $\TT$ be a $\V$-monad on $\uV$, and let $A$ be a $\TT$-algebra.  Write $\T$ for the $\uV$-theory corresponding to $\TT$.  The \textbf{(absolute) commutant} of $A$ (or of $\TT$ with respect to $A$) is defined as the $\V$-monad $\TT^\perp_A$ corresponding to the commutant $\T^\perp_A$ of the (normal) $\T$-algebra $\T \rightarrow \uV$ corresponding to $A$.  Note that this commutant necessarily exists, by \bref{thm:cmtnt_alg_ele}.
\end{DefSub}

Here the notion of commutant intersects with the notion of \textit{codensity monad} \cite{Kock:CodMnd}:

\begin{PropSub}\label{thm:abs_cmtnt_of_talg_vs_codensity_mnd}
The absolute commutant $\TT^\perp_A$ of a $\V$-monad $\TT$ with respect to a $\TT$-algebra $A$ is the codensity $\V$-monad (see \textnormal{\cite[II]{Dub}}) of the $\T$-algebra $\widetilde{A}:\T \rightarrow \uV$ corresponding to $A$, where we denote by $\T$ the $\uV$-theory corresponding to $\TT$.
\end{PropSub}
\begin{proof}
By \cite[II.3]{Dub}, the $\uV$-algebraic structure $\textnormal{Str}(\widetilde{A})$ of $\widetilde{A}$ \pbref{rem:jalg_str} is the Kleisli $\uV$-theory $\uV_\SSS^\op$ of the codensity $\V$-monad $\SSS$ for $\widetilde{A}$, and in particular, $\SSS$ exists since $\textnormal{Str}(\widetilde{A})$ does.  In other words, $\textnormal{Str}(\widetilde{A})$ is the $\uV$-theory $\mathsf{t}(\SSS) = \uV_\SSS^\op$ corresponding to $\SSS$.  But as we noted in \bref{rem:jalg_str}, $\textnormal{Str}(\widetilde{A}) = \T^\perp_A$ in this case, and the $\V$-monad associated to this $\uV$-theory is therefore $\TT^\perp_A = \mathsf{m}(\T^\perp_A) = \mathsf{m}(\mathsf{t}(\SSS)) \cong \SSS$.
\end{proof}

\newcommand{\noopsort}[1]{}
\providecommand{\bysame}{\leavevmode\hbox to3em{\hrulefill}\thinspace}
\providecommand{\MR}{\relax\ifhmode\unskip\space\fi MR }
\providecommand{\MRhref}[2]{%
  \href{http://www.ams.org/mathscinet-getitem?mr=#1}{#2}
}
\providecommand{\href}[2]{#2}



\begin{thebibliography}{10}

\bibitem{Bir:SelPa}
G.~Birkhoff, \emph{Selected papers on algebra and topology}, Birkh\"auser
  Boston, Inc., Boston, MA, 1987.

\bibitem{BoDay}
F.~Borceux and B.~Day, \emph{Universal algebra in a closed category}, J. Pure
  Appl. Algebra \textbf{16} (1980), no.~2, 133--147.

\bibitem{DlRi}
V.~Dlab and C.~M. Ringel, \emph{Rings with the double centralizer property}, J.
  Algebra \textbf{22} (1972), 480--501.

\bibitem{Dub:EnrStrSem}
E.~J. Dubuc, \emph{Enriched semantics-structure (meta) adjointness}, Rev. Un.
  Mat. Argentina \textbf{25} (1970), 5--26.

\bibitem{Dub}
\bysame, \emph{Kan extensions in enriched category theory}, Lecture Notes in
  Mathematics, Vol. 145, Springer-Verlag, 1970.

\bibitem{EiKe}
S.~Eilenberg and G.~M. Kelly, \emph{Closed categories}, Proc. {C}onf.
  {C}ategorical {A}lgebra ({L}a {J}olla, {C}alif., 1965), Springer, 1966,
  pp.~421--562.

\bibitem{LfGa}
R.~Garner and I.~L\'opez~Franco, \emph{Commutativity}, J. Pure Appl. Algebra
  (2016), 1707--1751.

\bibitem{Ke:MonoEpiPb}
G.~M. Kelly, \emph{Monomorphisms, epimorphisms, and pull-backs}, J. Austral.
  Math. Soc. \textbf{9} (1969), 124--142.

\bibitem{Ke:FL}
\bysame, \emph{Structures defined by finite limits in the enriched context.
  {I}}, Cahiers Topologie G\'eom. Diff\'erentielle \textbf{23} (1982), no.~1,
  3--42.

\bibitem{Ke:Ba}
\bysame, \emph{Basic concepts of enriched category theory}, Repr. Theory Appl.
  Categ. (2005), no.~10, Reprint of the 1982 original [Cambridge Univ. Press].

\bibitem{Kock:CodMnd}
A.~Kock, \emph{Continuous {Y}oneda representation of a small category}, Aarhus
  University Preprint, 1966.

\bibitem{Kock:Comm}
\bysame, \emph{Monads on symmetric monoidal closed categories}, Arch. Math.
  (Basel) \textbf{21} (1970), 1--10.

\bibitem{Kock:DblDln}
\bysame, \emph{On double dualization monads}, Math. Scand. \textbf{27} (1970),
  151--165 (1971).

\bibitem{Law:PhD}
F.~W. Lawvere, \emph{Functorial semantics of algebraic theories}, Dissertation,
  Columbia University, New York. Available in: \emph{Repr. Theory Appl. Categ.}
  \textbf{5} (2004), 1963.

\bibitem{Lin:AutEqCats}
F.~E.~J. Linton, \emph{Autonomous equational categories}, J. Math. Mech.
  \textbf{15} (1966), 637--642.

\bibitem{Lin:Eq}
\bysame, \emph{Some aspects of equational categories}, Proc. {C}onf.
  {C}ategorical {A}lgebra ({L}a {J}olla, {C}alif., 1965), Springer, 1966,
  pp.~84--94.

\bibitem{Lin:OutlFuncSem}
\bysame, \emph{An outline of functorial semantics}, Sem. on {T}riples and
  {C}ategorical {H}omology {T}heory ({ETH}, {Z}\"urich, 1966/67), Springer,
  1969, pp.~7--52.

\bibitem{Lu:PhD}
R.~B.~B. Lucyshyn-Wright, \emph{{R}iesz-{S}chwartz extensive quantities and
  vector-valued integration in closed categories}, Ph.D. thesis, York
  University, {\noopsort{a}}2013,
  \href{http://arxiv.org/abs/1307.8088}{arXiv:1307.8088}.

\bibitem{Lu:CT2015}
\bysame, \emph{A general theory of measure and distribution monads founded on
  the notion of commutant of a subtheory}, Talk at \textit{Category Theory
  2015}, Aveiro, Portugal, June {\noopsort{b}}2015.

\bibitem{Lu:EnrAlgTh}
\bysame, \emph{Enriched algebraic theories and monads for a system of arities},
  Theory Appl. Categ. \textbf{31} ({\noopsort{c}}2016), 101--137.

\bibitem{Lu:CvxAffCmt}
\bysame, \emph{Convex spaces, affine spaces, and commutants for algebraic
  theories}, Appl. Categ. Structures ({\noopsort{d}}2017, in press), \href{https://doi.org/10.1007/s10485-017-9496-9}{doi:10.1007/s10485-017-9496-9}

\bibitem{Lu:FDistn}
\bysame, \emph{Functional distribution monads in functional-analytic contexts}, Adv. Math. (2017, to appear), also available as \href{https://arxiv.org/abs/1701.08152}{arXiv:1701.08152}.

\bibitem{Pow:EnrLaw}
J.~Power, \emph{Enriched {L}awvere theories}, Theory Appl. Categ. \textbf{6}
  (1999), 83--93.

\bibitem{Wra:AlgTh}
G.~C. Wraith, \emph{Algebraic theories}, Lectures Autumn 1969. Lecture Notes
  Series, No. 22, Matematisk Institut, Aarhus Universitet, Aarhus, 1970
  (Revised version 1975).

\end{thebibliography}
\end{document}